\documentclass[11pt]{amsproc}
\usepackage[british]{babel}

\usepackage{a4wide}
\usepackage{setspace}
\usepackage{amssymb}
\usepackage{amsmath}
\usepackage{amsthm}
\usepackage{array}
\usepackage{enumitem}
\usepackage[colorlinks,citecolor=blue, pagebackref]{hyperref}

\newtheorem{thm}{Theorem}[section]

\newtheorem{cor}[thm]{Corollary}
\newtheorem{lem}[thm]{Lemma}

\newtheorem{prop}[thm]{Proposition}

\theoremstyle{definition}

\newtheorem{exmpl}[thm]{Example}

\newtheorem{definition}[thm]{Definition}

\newtheorem{remark}[thm]{Remark}

\renewcommand{\epsilon}{\varepsilon}
\renewcommand{\phi}{\varphi}
\newcommand{\defeq}{\mathrel{\mathop:}=}                         
\newcommand{\eqdef}{\mathrel{\mathopen={\mathclose:}}}

\DeclareMathOperator{\St}{St}
\DeclareMathOperator{\dom}{dom}
\DeclareMathOperator{\spt}{spt}
\DeclareMathOperator{\diam}{diam}
\DeclareMathOperator{\Aut}{Aut}

\begin{document}
%%%%%%%%%%%%%%%%%%%%%%%%%%%%%%%%

\onehalfspace

\setlist{noitemsep}

\author{Friedrich Martin Schneider}
\address{F.M.S., Institut f\"ur Algebra, TU Dresden, 01062 Dresden, Germany }

\author{Andreas Thom}
\address{A.T., Institut f\"ur Geometrie, TU Dresden, 01062 Dresden, Germany }

\title{Topological matchings and amenability}
\date{\today}

\begin{abstract} 
  We establish a characterization of amenability for general Hausdorff topological groups in terms of matchings with respect to finite uniform coverings. Furthermore, we prove that it suffices to just consider two-element uniform coverings. We also show that extremely amenable as well as compactly approximable topological groups satisfy a perfect matching property condition -- the latter even with regard to arbitrary (i.e., possibly infinite) uniform coverings. Finally, we prove that the automorphism group of a Fra\"iss\'e limit of a metric Fra\"iss\'e class is amenable if and only if the considered metric Fra\"iss\'e class has a certain Ramsey-type matching property.
\end{abstract}

\maketitle

%%%%%%%%%%%%%%%%%%%%%%%%%%%%%%%%%%%%%%%%%%
%%%%%%%%%%%%%%%%%%%%%%%%%%%%%%%%%%%%%%%%%%

\tableofcontents

\section{Introduction}

In this paper we study the various characterizations of amenability for general topological groups.
A Hausdorff topological groups $G$ is said to be \emph{amenable} if there exists a left-invariant mean on the algebra of bounded, uniformly continuous, real-valued functions on $G$. In case $G$ is discrete, this corresponds to the existence of a left-invariant finitely additive probability measure on $G$, and a classical result \cite{folner} characterizes amenability for discrete groups by the existence of so-called \emph{F\o lner sets} -- finite subsets of the group that are almost invariant with respect to a finite set of translations. Our motivation to write this note was to provide a F\o lner-type characterization of amenability for general Hausdorff topological groups. Indeed, we show that a topological group $G$ is amenable if and only if, for every finite uniform covering of $G$ and every finite subset $E$ of $G$, there exists a finite non-empty subset $F$ of $G$ such that every $E$-translate of $F$ can be \emph{almost matched with respect to the uniform covering} with $F$ (see Theorem~\ref{theorem:amenable.groups}). By a matching of subsets with respect to a covering we mean a bijection which respects the covering. Along the way towards Theorem~\ref{theorem:amenable.groups}, we also prove an analogous, but slightly more general characterization of amenability for perfect Hausdorff uniform dynamical systems (see Corollary~\ref{corollary:amenable.perfect.hausdorff.dynamical.systems}).

Furthermore, we show that a topological group $G$ is amenable if and only if it satisfies the matching condition above with respect to all two-element uniform coverings (see Theorem~\ref{theorem:two.element.coverings}) -- generalizing a result of Moore  \cite{moore}. As a consequence, it follows that $G$ is amenable if and only if every single bounded uniformly continuous real-valued function on $G$ can be averaged invariantly (see Corollary~\ref{corollary:measuring.single.functions}), generalizing a result of \cite{kaichouh}. We also reformulate our matching conditions for non-archimedean groups in terms of coset colorings (see Corollary~\ref{coro:totally.disconnected}). As an application, we show that a non-archimedean Hausdorff topological group $G$ is amenable if and only if every minimal sub-flow of the canonical $G$-flow on $2^{H\backslash G}$ for an open subgroup $H \leq G$ is amenable (see Corollary~\ref{corollary:non.archimedean}). In the classical case of discrete groups discussed above, our results provide a characterization of amenability by means of a weak form of F\o lner sets (see Corollary~\ref{coro.discrete}), which was proven recently by Moore \cite{moore}. 

Moreover, we establish certain perfect matching results with respect to uniform coverings of compactly approximable and extremely amenable groups (see Proposition~\ref{proposition:compactly.approximable} and Corollary~\ref{corollary:extremely.amenable.groups}), which yields a characterization of injectivity of von Neumann algebra in terms of a combinatorial property of its unitary group. Finally, we draw a connection to continuous logic by showing that the automorphism group of a Fra\"iss\'e limit of a metric Fra\"iss\'e class is amenable if and only if the considered metric Fra\"iss\'e class satisfies a certain Ramsey-type matching condition (see Theorem~\ref{theorem:metric.ramsey}).

\vspace{0.2cm}

The paper is organized as follows. Section~\ref{section:uniform.spaces} is supposed to provide some background regarding uniform spaces. In Section~\ref{section:amenability} we recall basic notions and facts concerning means on function spaces and then discuss the concept of amenability for dynamical systems in general and for topological groups in particular. Section~\ref{section:matchings} gives a brief reminder on matchings in bipartite graphs, including Hall's marriage theorem. In Section~\ref{section:matchings.in.dynamical.systems} we prove a characterization of amenability for perfect Hausdorff uniform dynamical systems in terms of matchings with respect to finite uniform coverings. In Section~\ref{section:matchings.in.topological.groups} we prove the aforementioned amenability criteria for general Hausdorff topological groups. In Section~\ref{section:colorings.of.non.archimedean.groups} we revisit the established matching criteria for non-archimedean groups. In Section~\ref{section:strong.matching.conditions} we discuss several strengthened matching conditions satisfied by compact, compactly approximable, or extremely amenable groups, respectively. Finally, in Section~\ref{section:ramsey.theory} we prove the mentioned correspondence between a certain Ramsey-type matching property for a metric Fra\"iss\'e class and the amenability of the automorphism group of its Fra\"iss\'e limit.

\section{Uniform spaces and their automorphism groups}\label{section:uniform.spaces}

In this section we shall recall the very basics concerning uniform spaces -- also in order to keep this paper reasonably self-contained. We will follow the approach of \cite{Isbell}.

In order to introduce the concept of a uniform space, we shall need some set-theoretic basics. Let $X$ be a set. We denote by $\mathcal{P}(X)$ the set of all subsets of $X$. Let $\mathcal{U},\mathcal{V} \subseteq \mathcal{P}(X)$. We say that $\mathcal{V}$ \emph{refines} $\mathcal{U}$ and write $\mathcal{U} \preceq \mathcal{V}$ if \begin{displaymath}
	\forall V \in \mathcal{V} \, \exists U \in \mathcal{U} \colon \, V \subseteq U .
\end{displaymath} Furthermore, let $\mathcal{U} \wedge \mathcal{V} \defeq \mathcal{U} \cup \mathcal{V}$ and $\mathcal{U} \vee \mathcal{V} \defeq \{ U \cap V \mid U \in \mathcal{U}, \, V \in \mathcal{V} \}$. More generally, if $(\mathcal{U}_{i})_{i \in I}$ is a family of subsets of $\mathcal{P}(X)$, then we define $\bigwedge\nolimits_{i \in I} \mathcal{U}_{i} \defeq \bigcup\nolimits_{i \in I} \mathcal{U}_{i}$ and \begin{displaymath}
	\left. \bigvee\nolimits_{i \in I} \mathcal{U}_{i} \defeq \left\{ \bigcap\nolimits_{i \in I} U_{i} \, \right| (U_{i})_{i \in I} \in \prod\nolimits_{i \in I} \mathcal{U}_{i} \right\} .
\end{displaymath} For a subset $S \subseteq X$, we call $\St (S,\mathcal{U}) \defeq \bigcup \{ U \in \mathcal{U} \mid U \cap S \ne\emptyset \}$ the \emph{star} of $S$ with respect to $\mathcal{U}$. Likewise, given any $x \in X$, we call $\St (x,\mathcal{U}) \defeq \St (\{ x \},\mathcal{U})$ the \emph{star} of $x$ with respect to $\mathcal{U}$. Moreover, the \emph{star} of $\mathcal{U}$ is defined to be $\mathcal{U}^{\ast} \defeq \{ \St (U,\mathcal{U}) \mid U \in \mathcal{U} \}$. Besides, let $\mathcal{U}^{\ast, 0} = \mathcal{U}$ and $\mathcal{U}^{\ast, n+1} \defeq (\mathcal{U}^{\ast, n})^{\ast}$ for every $n \in \mathbb{N}$. We say that $\mathcal{V}$ is a \emph{star-refinement} of $\mathcal{U}$ and write $\mathcal{U} \preceq^{\ast} \mathcal{V}$ if $\mathcal{U} \preceq \mathcal{V}^{\ast}$. We shall call $\mathcal{U}$ a \emph{covering} of $X$ if $X = \bigcup \mathcal{U}$. We denote by $\mathcal{C}(X)$ the set of all coverings of $X$. A \emph{uniformity} on $X$ is a non-empty subset $\mathcal{D} \subseteq \mathcal{C}(X)$ such that \begin{enumerate}
	\item[(1)] 	$\forall \mathcal{U} \in \mathcal{D} \, \forall \mathcal{V} \in \mathcal{C}(X) \colon \, \mathcal{V} \preceq \mathcal{U} \Longrightarrow \mathcal{V} \in \mathcal{D}$,
	\item[(2)]	$\forall \mathcal{U},\mathcal{V} \in \mathcal{D} \, \exists \mathcal{W} \in \mathcal{D} \colon \, \mathcal{U} \preceq^{\ast} \mathcal{W} , \, \mathcal{V} \preceq^{\ast} \mathcal{W}$.
\end{enumerate}

Now we come to uniform spaces. A \emph{uniform space} is a non-empty set $X$ equipped with a uniformity on $X$, whose elements are called the \emph{uniform coverings} of the uniform space $X$. Let $X$ be a uniform space. The set of all finite uniform coverings of $X$ shall be denoted by $\mathcal{N}(X)$. The \emph{topology of $X$} is defined as follows: a subset $S \subseteq X$ is \emph{open} in $X$ if, for every $x \in S$, there exists a uniform covering $\mathcal{U}$ of $X$ such that $\St (x,\mathcal{U}) \subseteq S$. Let $Y$ be another uniform space. A map $f \colon X \to Y$ is said to be \emph{uniformly continuous} if $f^{-1}(\mathcal{U}) \defeq \{ f^{-1}(U) \mid U \in \mathcal{U} \}$ is a uniform covering of $X$ whenever $\mathcal{U}$ is a uniform covering of $Y$. We denote by ${\rm UC}(X,Y)$ the set of all uniformly continuous functions from $X$ to $Y$. A bijection $f \colon X \to Y$ is called an \emph{isomorphism} if both $f$ and $f^{-1}$ are uniformly continuous maps. By an \emph{automorphism of $X$}, we mean an isomorphism from $X$ to itself. The automorphism group of $X$ shall be denoted by $\Aut (X)$. Note that any uniformly continuous map between uniform spaces is continuous with regard to the respective topologies.

It is well known that any metric space constitutes a uniform space: if $X$ is a metric space, then we may consider $X$ as a uniform space by equipping it with the \emph{induced uniformity}, that is, $\{ \mathcal{U} \subseteq \mathcal{P}(X) \mid \exists r > 0 \colon \, \mathcal{U} \preceq \{ B(x,r) \mid x \in X \} \}$. This particularly applies to the space of real numbers. Concerning a uniform space $X$, we denote by ${\rm UC}(X)$ the set of all uniformly continuous functions from $X$ to $\mathbb{R}$, and we put ${\rm UC}_{b}(X) \defeq {\rm UC}(X) \cap \ell^{\infty}(X)$.

Another example of uniform spaces is provided by the class of compact Hausdorff spaces. In fact, if $X$ is a compact Hausdorff space, then $\{ \mathcal{U} \subseteq \mathcal{P}(X) \mid \exists \mathcal{V} \in \mathcal{C}(X) \text{ open}\colon \, \mathcal{U} \preceq \mathcal{V} \}$ is the unique uniformity on $X$ inducing the topology of $X$. In particular, if $X$ is metrizable, then the uniformity above coincides with any uniformity on $X$ induced by a metric generating the topology of $X$. Furthermore, a mapping from a compact Hausdorff space into any uniform space is continuous if and only if it is uniformly continuous. For further reading about the uniform structure of compact Hausdorff spaces, we refer to \cite{Isbell}.

\begin{lem}\label{lemma:uniform.tiles} Let $X$ be a uniform space, let $H \subseteq {\rm UC}_{b}(X)$ be finite and $\epsilon > 0$. Then there exists $\mathcal{U} \in \mathcal{N}(X)$ such that $\diam f(U) \leq \epsilon$ for all $U \in \mathcal{U}$ and $f \in H$. \end{lem}

\begin{proof} Consider the compact metric space $K \defeq \prod_{f \in H} \overline{f(X)}$ equipped with the usual Euclidean metric. There exists a finite open covering $\mathcal{V}$ of $K$ such that $\diam (V) \leq \epsilon$ for all $V \in \mathcal{V}$. Since $H$ is a set of uniformly continuous functions, $\Phi \colon X \to K, \, x \mapsto (f(x))_{f \in H}$ is uniformly continuous as well. Hence, $\mathcal{U} \defeq \Phi^{-1}(\mathcal{V})$ is a finite uniform covering of $X$. Besides, $\diam (f(\Phi^{-1}(V))) \leq \diam (V) \leq \epsilon$ for all $V \in \mathcal{V}$. Hence, $\diam (U) \leq \epsilon$ for all $U \in \mathcal{U}$. \end{proof}

We shall need some further observations concerning finite uniform coverings.

\begin{lem}[\cite{Isbell}]\label{lemma:uniform.star.refinements} Let $X$ be a uniform space. If $\mathcal{U}$ is a finite uniform covering of $X$, then there exists a finite open uniform covering $\mathcal{V}$ of $X$ such that $\mathcal{U} \preceq^{\ast} \mathcal{V}$. \end{lem}

Among all finite coverings of a uniform space, the uniform ones are exactly those which admit a subordinate uniform partition of the unity. To make this precise, let us agree on some additional notation: if $X$ is a topological space, then we define $\spt (f) \defeq \overline{\{ x \in X \mid f(x) \ne 0 \}}$ for every continuous function $f \colon X \to \mathbb{R}$.

\begin{lem}[\cite{Isbell}]\label{lemma:uniform.covers} Let $X$ be a uniform space. A finite covering $\mathcal{U}$ of $X$ is uniform if and only if there exists a family of uniformly continuous functions $f_{U} \colon X \to [0,1]$ ($U \in \mathcal{U}$) such that \begin{enumerate}
	\item[$(1)$]	$\spt (f_{U}) \subseteq U$ for every $U \in \mathcal{U}$,
	\item[$(2)$]	$\sum_{U \in \mathcal{U}} f_{U}(x) = 1$ for all $x \in X$.
\end{enumerate} \end{lem}

It is well known that any topological group may be considered as a uniform space. In order to explain this and to agree on some additional notation, let $G$ be an arbitrary topological group. We denote by $\mathcal{U}(G)$ the filter of all neighborhoods of the neutral element in $G$. We define $G_{r}$ to be the uniform space obtained by endowing $G$ with the \emph{right uniformity}, i.e., \begin{displaymath}
	\{ \mathcal{U} \subseteq \mathcal{P}(G) \mid \exists U \in \mathcal{U}(G) \colon \, \mathcal{U} \preceq \{ Ux \mid x \in G \} \} ,
\end{displaymath} and we denote by $G_{l}$ the uniform space consisting of $G$ along with the \emph{left uniformity}, i.e., \begin{displaymath}
	\{ \mathcal{U} \subseteq \mathcal{P}(G) \mid \exists U \in \mathcal{U}(G) \colon \, \mathcal{U} \preceq \{ xU \mid x \in G \} \} .
\end{displaymath} It is easy to see that the topology generated by each of these uniformities is precisely the original topology of $G$. Let us fix the following terminology with regard to an identity neighborhood $U$ in $G$: a set $\mathcal{U} \subseteq \mathcal{P}(G)$ is called a \emph{$U$-uniform covering of $G_{r}$} if $\mathcal{U} \preceq \{ Ux \mid x \in G \}$, and $\mathcal{U}$ is called a \emph{$U$-uniform covering of $G_{l}$} if $\mathcal{U} \preceq \{ xU \mid x \in G \}$. Clearly, a set $\mathcal{U} \subseteq \mathcal{P}(G)$ is a uniform covering of $G_{r}$ ($G_{l}$, respectively) if and only if $\mathcal{U}$ is a $U$-uniform covering of $G_{r}$ ($G_{l}$, respectively) for some identity neighborhood $U$ in $G$. Besides, recall that any continuous homomorphism from $G$ into another topological group $H$ constitutes a uniformly continuous map both from $G_{r}$ to $H_{r}$ and from $G_{l}$ to $H_{l}$. In the following, we shall mainly be concerned with the right uniformity -- except for Section~\ref{section:ramsey.theory}. However, note that $\iota \colon G_{r} \to G_{l}, \, x \mapsto x^{-1}$ is an isomorphism of uniform spaces. Hence, any statement about the right uniformity can be translated into an equivalent statement about the left uniformity in straight-forward manner. For more details concerning uniform structures on topological groups, we refer to \cite{roelcke}.

Furthermore, we need to recall the concept of uniform convergence. To this end, let $X,Y$ be uniform spaces. Concerning a function $f \in {\rm UC}(X,Y)$ and a uniform covering $\mathcal{U}$ of $Y$, we define $[f,\mathcal{U}] \defeq \{ g \in {\rm UC}(X,Y) \mid \forall x \in X \colon \, g(x) \in \St (f(x),\mathcal{U}) \}$. It is straightforward to check that $\{ \mathcal{V} \subseteq \mathcal{P}({\rm UC}(X,Y)) \mid \exists \mathcal{U} \textnormal{ uniform covering of } Y \colon \, \mathcal{V} \preceq \{ [f,\mathcal{U}] \mid f \in {\rm UC}(X,Y) \} \}$ constitutes uniformity on ${\rm UC}(X,Y)$, which we refer to as the \emph{uniformity of uniform convergence}. The induced topology on ${\rm UC}(X,Y)$ is called the \emph{topology of uniform convergence}. It is now straight-forward to check that $\Aut (X)$ endowed with the topology of uniform convergence constitutes a topological group, and that the corresponding right uniformity of this topological group is just the uniformity of uniform convergence on $\Aut (X)$. Note that if $G$ is a topological group, then the injective group homomorphism $\lambda_{G} \colon G \to \Aut (G_{r})$ given by $\lambda_{G}(g) (x) \defeq gx$ for $g,x \in G$ is continuous.

Suppose that $X$ is a compact Hausdorff space and $G$ is a topological group. Consider a \emph{$G$-flow} on $X$, i.e., a continuous homomorphism $\alpha \colon G \to \Aut (X)$. A subset $Y \subseteq X$ is called \emph{$\alpha$-invariant} if $\alpha(g)(Y) = Y$ for every $g \in G$. We say that $\alpha$ is \emph{minimal} if $\emptyset$ and $X$ are the only closed, $\alpha$-invariant subsets of $X$. By a \emph{subflow} of $\alpha$ we mean any flow of the form $\beta \colon G \to \Aut (Y), \, g \mapsto \alpha (g)|_{Y}$ where $Y$ is a closed, $\alpha$-invariant subset of $X$.

\section{Means and amenability}\label{section:amenability}

In this section we want to recall some general basics concerning means on function spaces. For this purpose, we follow the presentation in \cite{AnalysisOnSemigroups}. Furthermore, we briefly discuss the concept of amenability for dynamical systems in general and for topological groups in particular. For a more elaborate study of amenable topological groups including interesting examples, we refer to \cite{paterson,paterson92,runde,brownc,GrigorchukDeLaHarpe}.

For a start, we clarify some notation and recall some basic terminology regarding function spaces. Let $X$ be a set. For convenience, we shall denote by $\mathcal{F}(X)$ the set of all finite subsets of $X$. Additionally, we abbreviate $\mathcal{F}_{+}(X) \defeq \mathcal{F}(X)\setminus \{ \emptyset \}$. Furthermore, we consider the set \begin{displaymath}
	\Delta (X) \defeq \left\{ \delta \in [0,1]^{X} \left| \, \delta^{-1}((0,1]) \text{ finite, } \sum_{x \in X} \delta (x) = 1 \right\} \right.
\end{displaymath} of formal convex combinations over $X$. For every $\delta \in \Delta (X)$, we define $\spt (\delta) \defeq \delta^{-1}((0,1])$. As usual, for $x \in X$, define $\delta_{x} \in \Delta (X)$ by \begin{displaymath}
	\delta_{x}(y) \defeq \begin{cases}
		1 & \text{if } x=y ,\\
		0 & \text{otherwise}
	\end{cases} \qquad (y \in X) .
\end{displaymath} If $F$ is a finite non-empty subset of $X$, then we put $\delta_{F} \defeq \sum_{x \in F} \delta_{x} \in \Delta (X)$. Let $H$ be a linear subspace of $\ell^{\infty}(X)$. A \emph{mean} on $H$ is a linear map $\mu \colon H \to \mathbb{R}$ such that \begin{displaymath}
	\inf \{ f(x) \mid x \in X \} \leq \mu(f) \leq \sup \{ f(x) \mid x \in X \}
\end{displaymath} for all $f \in H$. The set of all means on $H$ is denoted by $M(H)$. For each $x \in X$, we obtain a mean on $H$ by $\nu_{x} \colon H \to \mathbb{R}, \, f \mapsto f(x)$. More generally, $\nu_{\delta} \defeq \sum_{x \in X} \delta(x)\nu_{x} \colon H \to \mathbb{R}$ is a mean on $H$ for every $\delta \in \Delta (X)$. In particular, \begin{displaymath}
	\nu_{F} \defeq \nu_{\delta_{F}} = \frac{1}{\vert F \vert}\sum_{x \in F} \nu_{x} \colon H \to \mathbb{R}
\end{displaymath} is a mean on $H$ for any non-empty finite subset $F \subseteq X$. Evidently, $\nu_{x} = \nu_{\delta_{x}}$ for every $x \in X$, and $\{ \nu_{\delta} \mid \delta \in \Delta (X) \}$ is nothing but the convex envelope of $\{ \nu_{x} \mid x \in X \}$ in $M(H)$.

\begin{thm}[\cite{AnalysisOnSemigroups}]\label{theorem:general.means} Let $X$ be a set and let $H$ be a linear subspace of $\ell^{\infty}(X)$ containing the constant functions. Then $M(H)$ is convex and weak-* compact. Furthermore, the convex subset $\{ \nu_{\delta} \mid \delta \in \Delta (X) \}$ is weak-* dense in $M(H)$. \end{thm}

Furthermore, let us point out the following modification of Theorem~\ref{theorem:general.means}.

\begin{lem}\label{lemma:dense.set.of.means} Suppose $X$ to be a topological space such that every open non-empty subset of $X$ is infinite. Let $H$ be a linear subspace of $C_{b}(X)$ containing the constant functions. Then $\{ \nu_{F} \mid F \in \mathcal{F}_{+}(X)\}$ is weak-* dense in $M(H)$. \end{lem}

\begin{proof} Let $\mu \in M(H)$, $H_{0} \in \mathcal{F}(H)$ and $\epsilon > 0$. We define $s \defeq \sup_{f \in H_{0}} \Vert f \Vert_{\infty} + 1$ and $\theta \defeq \frac{\epsilon}{3s}$. According to Theorem~\ref{theorem:general.means}, there exist $F \in \mathcal{F}_{+}(X)$ and $\alpha \colon F \to (0,1]$ such that $\sum_{x \in F} \alpha(x) = 1$ and $|\mu(f) - \sum_{x \in F} \alpha(x)f(x)| \leq \frac{\epsilon}{3}$ for all $f \in H_{0}$. It is well-known that $\{ \beta \in ((0,1] \cap \mathbb{Q})^{F} \mid \sum_{x \in F} \beta (x) = 1 \}$ is dense in $\{ \beta \in (0,1]^{F} \mid \sum_{x \in F} \beta (x) = 1 \}$. Hence, there exists $\beta \colon F \to (0,1] \cap \mathbb{Q}$ such that $\sum_{x \in F} \beta(x) = 1$ and $\sum_{x \in F} |\alpha(x) - \beta(x)| \leq \theta$. The latter assertion readily implies that \begin{displaymath}
		\left|\sum_{x \in F} \alpha (x) f(x) - \sum_{x \in F} \beta (x) f(x)\right| \leq \sum_{x \in F} |\alpha (x) - \beta(x)| \Vert f \Vert_{\infty} \leq \frac{\epsilon}{3}
	\end{displaymath} for each $f \in H_{0}$. There exist $n \in \mathbb{N}\setminus \{ 0 \}$ and $\gamma \colon F \to \{ 1,\ldots ,n \}$ such that $\beta(x) = \frac{\gamma(x)}{n}$ for all $x \in F$. Now, if $x \in F$, then $V(x) \defeq \bigcap \{ f^{-1}((f(x)-\frac{\epsilon}{3},f(x)+\frac{\epsilon}{3})) \mid f \in H_{0} \}$ is an open non-empty subset of $X$. Since every open non-empty subset of $X$ is infinite, there is $\Phi \colon F \to \mathcal{F}_{+}(X)$ such that \begin{enumerate}
	\item[(1)]	$\Phi(x) \subseteq V(x)$ for every $x \in F$,
	\item[(2)]	$|\Phi(x)| = \gamma(x)$ for every $x \in F$,
	\item[(3)]	$\Phi(x) \cap \Phi(y) = \emptyset$ for any two distinct $x,y \in F$.
\end{enumerate} Let $E \defeq \bigcup \{ \Phi(x) \mid x \in F \}$. We observe that $|E| = n$. For every $f \in H_{0}$, it follows that \begin{align*}
\left| \sum_{x \in F} \beta (x) f(x) - \frac{1}{|E|}\sum_{y \in E} f(y) \right| &= \frac{1}{n} \left| \sum_{x \in F} \gamma (x) f(x) - \sum_{x \in F}\sum_{y \in \Phi(x)} f(y) \right| \\
&\leq \frac{1}{n} \sum_{x \in F} \left| \gamma (x) f(x) - \sum_{y \in \Phi (x)} f(y) \right| \\
&\leq \frac{1}{n} \sum_{x \in F} \sum_{y \in \Phi (x)} |f(x) - f(y)| \\
&\leq \frac{1}{n} \sum_{x \in F} \frac{\epsilon \gamma (x)}{3} = \frac{\epsilon}{3}
\end{align*} and therefore \begin{align*}
\left| \mu(f) - \nu_{E}(f) \right| &\leq \left| m(f) - \sum_{x \in F} \alpha(x)f(x)\right| + \left|\sum_{x \in F} \alpha (x) f(x) - \sum_{x \in F} \beta (x) f(x)\right| \\
& \, \, \, \, \, \, + \left| \sum_{x \in F} \beta (x) f(x) - \frac{1}{|E|}\sum_{y \in E} f(y) \right| \\
&\leq \, \frac{\epsilon}{3} + \frac{\epsilon}{3} + \frac{\epsilon}{3} = \epsilon
\end{align*} This finishes the proof. \end{proof}

Recall that a topological space is \emph{perfect} if it does not contain any isolated points. Furthermore, a topological space $X$ is called \emph{homogeneous} if, for any two points $x,y \in X$, there exists a homeomorphism $g \colon X \to X$ such that $g(x) = y$. It is easy to see the following:

\begin{remark}\label{remark:isolated.points} Let $X$ be a topological space. The following statements hold. \begin{enumerate}
		\item[(1)]	If $X$ is homogeneous and not discrete, then $X$ is perfect.
		\item[(2)]	If $X$ is $T_{1}$ and perfect, then every open non-empty subset of $X$ is infinite.
	\end{enumerate} \end{remark}

Now we come to amenability. To this end, let $X$ be a uniform space. A \emph{mean} on $X$ is a mean on ${\rm UC}_{b}(X)$. We denote by $M(X)$ the set of all means on $X$. Consider a subgroup $G$ of $\Aut (X)$. We refer to the pair $(X,G)$ as a \emph{dynamical system}. An \emph{invariant mean} on $(X,G)$ is a mean $\mu$ on $X$ such that $\mu(f) = \mu(f \circ g)$ for all $f \in {\rm UC}_{b}(X)$ and $g \in G$. The set of invariant means on $(X,G)$ shall be denoted by $M(X,G)$. We call $(X,G)$ \emph{amenable} if $M(X,G) \ne \emptyset$. Given any topological group $G$ and some compact Hausdorff space $X$, we say that a $G$-flow $\alpha \colon G \to \Aut (X)$ on $X$ is \emph{amenable} if the dynamical system $(X,\alpha (G))$ is amenable.

As pointed out at the end of Section~\ref{section:uniform.spaces}, any topological group may be considered as a uniform space, wherefore the previous definition particularly applies to topological groups. To elaborate on this case, let $G$ be a topological group. An \emph{invariant mean} on $G$ is an invariant mean of the dynamical system $(G_{r},\lambda_{G}(G))$. Accordingly, we call $G$ \emph{amenable} if the dynamical system $(G_{r},\lambda_{G}(G))$ is amenable, i.e., there is a mean $\mu$ on ${\rm UC}_{b}(G_{r})$ such that $\mu(f) = \mu(f \circ \lambda_{G}(g))$ for all $f \in {\rm UC}_{b}(G_{r})$ and $g \in G$. It is well known that $G$ is amenable if and only if every $G$-flow is amenable. Moreover, let us recall the following well-known characterization of amenability for discrete groups.

\begin{thm}[\cite{folner}]\label{theorem:folner} A discrete group $G$ is amenable if and only if, for all $\theta \in [0,1)$ and $E \in \mathcal{F}(G)$, there is some $F \in \mathcal{F}_{+}(G)$ such that $|F \cap gF| \geq \theta |F|$ for all $g \in E$. \end{thm}

For later use, we equip the set of means on a given topological group with a suitable semigroup structure by extending the group multiplication in the usual way. So, let $G$ be a topological group. It is easy to see that the function $f_{\nu} \colon G \to \mathbb{R}, \, g \mapsto \nu (f \circ \lambda_{G}(g))$ is member of ${\rm UC}_{b}(G_{r})$ whenever $f \in {\rm UC}_{b}(G_{r})$ and $\nu \in M(G_{r})$. Hence, we may define $(\mu \nu ) (f) \defeq \mu (f_{\nu})$ for any two $\mu, \nu \in M(G_{r})$. Note that this multiplication is associative and hence turns $M(G_{r})$ into a semigroup. Besides, we observe that $\Delta (G)$ also carries a natural semigroup structure given by the multiplication \begin{displaymath}
	(\alpha \beta )(g) \defeq \sum_{h \in G} \alpha (gh^{-1})\beta (h) \qquad (\alpha,\beta \in \Delta(X)) ,
\end{displaymath} and both $G \to \Delta (G), \, g \mapsto \delta_{g}$ and $\Delta (G) \to M(G_{r}), \, \delta \mapsto \nu_{\delta}$ are homomorphisms. Furthermore, note that $\spt (\alpha \beta) \subseteq (\spt \alpha)(\spt \beta)$ for any two elements $\alpha,\beta \in \Delta(G)$.

\section{Matchings in bipartite graphs}\label{section:matchings}

In this section we briefly recall basic notions and facts about matchings in bipartite graphs. This particularly includes Hall's marriage theorem (see Theorem~\ref{theorem:hall}). 

For a start, we clarify some terminology and notation. Let $\mathcal{B} = (X,Y,R)$ be a \emph{bipartite graph}, i.e., a triple consisting of two finite sets $X$ and $Y$ and a relation $R \subseteq X \times Y$. If $S \subseteq X$, then we define $N_{\mathcal{B}}(S) \defeq \{ y \in Y \mid \exists x \in S \colon (x,y) \in R \}$. A \emph{matching} in $\mathcal{B}$ is an injective map $\phi \colon D \to Y$ such that $D \subseteq X$ and $(x,\phi(x)) \in R$ for all $x \in D$. A matching $\phi$ in $\mathcal{B}$ is said to be \emph{perfect} if $\dom (\phi) = X$. Furthermore, we call \begin{displaymath}
	\mu (\mathcal{B}) \defeq \sup \{ |\dom \phi | \mid \phi \textnormal{ matching in } \mathcal{B} \}
\end{displaymath} the \emph{matching number} of $\mathcal{B}$. For later use we note the following simple observation.

\begin{remark}\label{remark:bijective.graph.homomorphisms} Let $\mathcal{B}_{0} = (X_{0},Y_{0},R_{0})$ and $\mathcal{B}_{1} = (X_{1},Y_{1},R_{1})$ be bipartite graphs. Suppose that there exist bijective maps $\phi \colon X_{0} \to X_{1}$ and $\psi \colon Y_{0} \to Y_{1}$ such that $(\phi(x),\psi(y)) \in R_{1}$ for all $(x,y) \in R_{0}$. Then $\mu (\mathcal{B}_{0}) \leq \mu (\mathcal{B}_{1})$. \end{remark}

We will need Hall's well-known matching theorem, which we restate for convenience.

\begin{thm}[\cite{Hall35}, \cite{Ore}]\label{theorem:hall} If $\mathcal{B} = (X,Y,R)$ is a bipartite graph, then \begin{displaymath}
	\nu (\mathcal{B}) = |X| - \sup \{ |S| - |N_{\mathcal{B}}(S)| \mid S \subseteq X \} .
\end{displaymath} \end{thm}

\begin{cor}\label{corollary:hall} A bipartite graph $\mathcal{B} = (X,Y,R)$ admits a perfect matching if and only if $|S| \leq |N_{\mathcal{B}}(S)|$ for every subset $S \subseteq X$. \end{cor}

In what follows, we shall have a closer look at bipartite graphs arising from uniform coverings of uniform spaces. For this purpose, we need to introduce some additional notation. Consider an arbitrary set $X$, finite subsets $E,F \subseteq X$, and a covering $\mathcal{U}$ of $X$. Then we define the bipartite graph \begin{displaymath}
		\mathcal{B} (E,F,\mathcal{U}) \defeq (E,F,R(E,F,\mathcal{U}))
\end{displaymath} with the relation given as follows: \begin{displaymath}
	R(E,F,\mathcal{U}) \defeq \{ (x,y) \in E \times F \mid y \in \St (x,\mathcal{U}) \} = \{ (x,y) \in E \times F \mid \exists U \in \mathcal{U} \colon \, \{ x,y \} \subseteq U \}.
\end{displaymath} Furthermore, we define $\mu (E,F,\mathcal{U}) \defeq \mu (\mathcal{B}(E,F,\mathcal{U}))$. Evidently, the following holds.

\begin{remark}\label{remark:estimate.for.matching.number} Let $\mathcal{U}$ be a covering of a set $X$. If $E,F \in \mathcal{F}(X)$, then $E \cap F \leq \mu (E,F,\mathcal{U})$. \end{remark}

The subsequent observations will prove useful in Section~\ref{section:matchings.in.dynamical.systems}.

\begin{lem}\label{lemma:composition.of.matchings} Let $\mathcal{U}$ be a covering of a set $X$ and let $F_{0},F_{1},F_{2} \in \mathcal{F}(X)$. Then \begin{displaymath}
	\mu (F_{0},F_{2},\mathcal{U}^{\ast}) \geq \mu (F_{0},F_{1},\mathcal{U}) + \mu (F_{1},F_{2},\mathcal{U}) - |F_{1}| .
\end{displaymath} \end{lem}

\begin{proof} Suppose $\phi_{0}$ and $\phi_{1}$ to be matchings in $\mathcal{B}(F_{0},F_{1},\mathcal{U})$ and $\mathcal{B}(F_{1},F_{2},\mathcal{U})$ such that $\vert \dom (\phi_{0}) \vert = \mu (F_{0},F_{1},\mathcal{U})$ and $\vert \dom (\phi_{1}) \vert = \mu (F_{1},F_{2},\mathcal{U})$, respectively. Put $D_{i} \defeq \dom (\phi_{i})$ for each $i \in \{ 0,1 \}$. Let $D \defeq \phi_{0}^{-1}(D_{1})$ and define $\psi \colon D \to F_{2}, \, x \mapsto \phi_{1}(\phi_{0}(x))$. Evidently, $\psi$ is injective. Besides, $\psi (x) = \phi_{1}(\phi_{0}(x)) \in \St (\phi_{0}(x),\mathcal{U}) \subseteq \St (x,\mathcal{U}^{\ast})$ for every $x \in D$. Hence, $\psi$ is a matching in $\mathcal{B}(F_{0},F_{2},\mathcal{U}^{\ast})$. Furthermore, \begin{align*}
	|F_{1}| - |D| &= |F_{1}| - |\phi_{0}(D)| = |F_{1}\setminus \phi_{0}(D)| = |F_{1} \setminus (\phi_{0}(D_{0}) \cap D_{1})| \\
	&= |(F_{1} \setminus \phi_{0}(D_{0})) \cup (F_{1}\setminus D_{1})| \leq |F_{1} \setminus \phi_{0}(D_{0})| + |F_{1}\setminus D_{1}| \\
	&= 2|F_{1}| - \mu (F_{0},F_{1},\mathcal{U}) - \mu (F_{1},F_{2},\mathcal{U}) 
\end{align*} and thus $\mu (F_{0},F_{2},\mathcal{U}^{\ast}) \geq |D| \geq \mu (F_{0},F_{1},\mathcal{U}) + \mu (F_{1},F_{2},\mathcal{U}) - |F_{1}|$. \end{proof}

\begin{cor}\label{corollary:composition.of.matchings} Let $\mathcal{U}$ be a covering of a set $X$ and let $F_{0},\ldots,F_{n} \in \mathcal{F}(X)$. Then \begin{displaymath}
	\mu (F_{0},F_{n},\mathcal{U}^{\ast, n-1}) \geq \sum_{i=0}^{n-1} \mu (F_{i},F_{i+1},\mathcal{U}) - \sum_{i = 1}^{n-1} |F_{i}| .
\end{displaymath} \end{cor}

\section{Matchings in dynamical systems}\label{section:matchings.in.dynamical.systems}

In this section we prove several characterizations for amenability of dynamical systems in terms of topological matchings.

\begin{thm}\label{theorem:vanishing.matching.number.implies.amenability} Let $(X,G)$ be a dynamical system. If \begin{displaymath}
	\inf_{E \in \mathcal{F}(G)} \inf_{\mathcal{U} \in \mathcal{N}(X)} \sup_{F \in \mathcal{F}_{+}(X)} \inf_{g \in E} \frac{\mu (F,g(F),\mathcal{U})}{|F|} = 1 ,
\end{displaymath} then $(X,G)$ is amenable. \end{thm}

\begin{proof} To prove the first implication, let $\epsilon > 0$, $H \in \mathcal{F}({\rm UC}_{b}(X))$ and $E \in \mathcal{F}_{+}(G)$. We observe that \begin{displaymath}
	A(H,E,\epsilon) \defeq \{ \nu \in M({\rm UC}_{b}(X)) \mid \forall f \in H \, \forall g \in E \colon \vert \nu(f) - \nu(f \circ g) \vert \leq \epsilon \}
\end{displaymath} is closed in the compact Hausdorff space $M({\rm UC}_{b}(X))$. We shall prove that $A(H,E,\epsilon) \ne \emptyset$. To this end, we put $\theta \defeq \epsilon /(1 + 2\sup_{f \in H} \Vert f \Vert_{\infty})$. Due to Lemma~\ref{lemma:uniform.tiles}, there exists $\mathcal{U} \in \mathcal{N}(X)$ such that $\diam f(U) \leq \theta$ for all $U \in \mathcal{U}$ and $f \in H$. By assumption, there exists $F \in \mathcal{F}_{+}(X)$ such that $|F| - \mu (F,g(F),\mathcal{U}) \leq \theta |F|$ for all $g \in E$. We show that $\nu_{F}$ is a member of $A(H,E,\epsilon)$. Of course, $\nu_{F} \in M({\rm UC}_{b}(X))$. Now, consider any $g \in E$. Let $\phi \colon D \to g(F)$ be an injective map such that $D \subseteq F$, $|D| = \mu (F,g(F),\mathcal{U})$, and $\phi (x) \in \St (x,\mathcal{U})$ for all $x \in D$. If $f \in H$, then \begin{align*}
	|\nu_{F}(f) - \nu_{F}(f \circ g)| &= \frac{1}{|F|}\left|\sum_{x \in F} f(x) - \sum_{x \in F} f(g(x))\right| \\
	&= \frac{1}{|F|}\left|\sum_{x \in D} (f(x)-f(\phi (x))) + \sum_{x \in F\setminus D} f(x) - \sum_{x \in g(F)\setminus \phi (D)} f(x)\right| \\
	&\leq \frac{1}{|F|}\left(\sum_{x \in D} |f(x)-f(\phi (x))| + \sum_{x \in F\setminus D} |f(x)| + \sum_{x \in g(F)\setminus \phi (D)} |f(x)|\right) \\
	&\leq \theta \frac{\mu (F,g(F),\mathcal{U})}{|F|} + 2\frac{|F| - \mu (F,g(F),\mathcal{U})}{|F|}\Vert f\Vert_{\infty} \\
	&\leq \theta + 2\Vert f\Vert_{\infty}\theta = (1+2\Vert f \Vert_{\infty}) \theta \leq \epsilon .
\end{align*} This proves our claim. Therefore, $A(H,E,\epsilon) \ne \emptyset$. Since \begin{displaymath}
	A(H_{0} \cup H_{1},E_{0} \cup E_{1}, \epsilon_{0} \wedge \epsilon_{1}) \subseteq A(H_{0},E_{0},\epsilon_{0}) \cap A(H_{1},E_{1},\epsilon_{1})
\end{displaymath} for all $H_{0},H_{1} \in \mathcal{F}({\rm UC}_{b}(X))$, $E_{0},E_{1} \in \mathcal{F}(G)$ and $\epsilon_{0},\epsilon_{1} > 0$, we conclude that \begin{displaymath}
	\mathcal{A} \defeq \{ A(H,E,\epsilon) \mid H \in \mathcal{F}({\rm UC}_{b}(X)), \, E \in \mathcal{F}(G), \, \epsilon > 0 \}
\end{displaymath} has the finite intersection property. By Theorem~\ref{theorem:general.means}, $M({\rm UC}_{b}(X))$ is compact. Consequently, $\bigcap \mathcal{A} \ne \emptyset$. Finally, we observe that $M(X,G) = \bigcap \mathcal{A}$, wherefore $(X,G)$ is amenable. \end{proof}

\begin{thm}\label{theorem:amenability.implies.vanishing.matching.number} Let $(X,G)$ be a dynamical system. If $(X,G)$ is amenable and every open non-empty subset of $X$ is infinite, then \begin{displaymath}
	\inf_{E \in \mathcal{F}(G)} \inf_{\mathcal{U} \in \mathcal{N}(X)} \sup_{F \in \mathcal{F}_{+}(X)} \inf_{g \in E} \frac{\mu (F,g(F),\mathcal{U})}{|F|} = 1 .
\end{displaymath} \end{thm}

\begin{proof} Let $\theta \in [0,1)$, $\mathcal{U} \in \mathcal{N}(X)$ and $E \in \mathcal{F}_{+}(G)$. By Lemma~\ref{lemma:uniform.covers}, there exists a family of uniformly continuous functions $f_{U} \colon X \to [0,1]$ ($U \in \mathcal{U}$) such that \begin{enumerate}
	\item[(1)]	$\spt (f_{U}) \subseteq U$ for every $U \in \mathcal{U}$,
	\item[(2)]	$\sum_{U \in \mathcal{U}} f_{U}(x) = 1$ for all $x \in X$.
\end{enumerate} Since $(X,G)$ is amenable and every open non-empty subset of $X$ is infinite, Lemma~\ref{lemma:dense.set.of.means} asserts that there exists $F \in \mathcal{F}_{+}(X)$ such that \begin{enumerate}
	\item[(3)]	$\frac{1}{|F|}\left|\sum_{x \in F} f_{U}(x) - \sum_{x \in F} f_{U}(g(x))\right| \leq \frac{1-\theta}{|\mathcal{U}|}$ for all $U \in \mathcal{U}$ and $g \in E$.
\end{enumerate} We show that $\mu (F,g(F),\mathcal{U}) \geq \theta |F|$ for all $g \in E$. To this end, let $g \in E$. We consider the bipartite graph $\mathcal{B} \defeq \mathcal{B}(F,g(F),\mathcal{U})$. If $S \subseteq F$, then we put $\mathcal{V} \defeq \{ U \in \mathcal{U} \mid U \cap S \ne \emptyset \}$ and $T \defeq N_{\mathcal{B}}(S)$, and we observe that \begin{align*}
	|S| &= \sum_{x \in S} 1 \stackrel{\textnormal{(2)}}{=} \sum_{x \in S} \sum_{U \in \mathcal{U}} f_{U}(x) \stackrel{\textnormal{(1)}}{=} \sum_{x \in S} \sum_{U \in \mathcal{V}} f_{U}(x) = \sum_{U \in \mathcal{V}} \sum_{x \in S} f_{U}(x) \leq \sum_{U \in \mathcal{V}} \sum_{x \in F} f_{U}(x) \\
	&\stackrel{\textnormal{(3)}}{\leq} \sum_{U \in \mathcal{V}} \left( \frac{(1-\theta ) |F|}{|\mathcal{U}| + 1} + \sum_{y \in g(F)} f_{U}(y) \right) \leq (1-\theta ) |F| + \sum_{U \in \mathcal{V}} \sum_{y \in g(F)} f_{U}(y) \\
	&\stackrel{\textnormal{(1)}}{=} (1-\theta ) |F| + \sum_{U \in \mathcal{V}} \sum_{y \in T} f_{U}(y) = (1-\theta ) |F| + \sum_{y \in T} \sum_{U \in \mathcal{V}} f_{U}(y) \\
	&\leq (1-\theta ) |F| + \sum_{y \in T} \sum_{U \in \mathcal{U}} f_{U}(y) \stackrel{\textnormal{(2)}}{=} (1-\theta ) |F| + \sum_{y \in T} 1 = (1-\theta )|F| + |T| ,
\end{align*} that is, $|S| - \left|N_{\mathcal{B}}(S)\right| \leq (1-\theta ) |F|$. According to Theorem~\ref{theorem:hall}, it follows that \begin{align*}
	\frac{\mu (F,g(F),\mathcal{U})}{|F|} &= \frac{|F|-\sup_{S \subseteq F} \left( |S| - \left| N_{\mathcal{B}}(S) \right|\right)}{|F|} \\
	&\geq \frac{|F|-(1-\theta )|F|}{|F|} = \theta . 
\end{align*} This substantiates that $\mu (X,G) = 1$. \end{proof}

\begin{cor}\label{corollary:amenable.perfect.hausdorff.dynamical.systems} Let $X$ be a perfect Hausdorff uniform space and let $G$ be a subgroup of $\Aut (X)$. Then $(X,G)$ is amenable if and only if the following holds: for every $\theta \in [0,1)$, every finite subset $E \subseteq G$, and every finite uniform covering $\mathcal{U}$ of $X$, there exists a finite non-empty subset $F \subseteq X$ such that \begin{displaymath}
	\forall g \in E \colon \, \mu (F,g(F),\mathcal{U}) \geq \theta \vert F \vert .
\end{displaymath} \end{cor}

\begin{proof} This follows from Theorem~\ref{theorem:vanishing.matching.number.implies.amenability}, Theorem~\ref{theorem:amenability.implies.vanishing.matching.number}, and Remark~\ref{remark:isolated.points}. \end{proof}

Of course, the previous result particularly applies to dynamical systems on perfect compact Hausdorff spaces. Moreover, since any open covering of a totally disconnected, compact Hausdorff space is refined by a finite partition of the space into clopen subsets, we also immediately obtain the following consequence.

\begin{cor}\label{corollary:amenable.cantor.dynamical.systems} Let $X$ be a perfect, totally disconnected, compact Hausdorff space, and let $G$ be a subgroup of $\Aut (X)$. Then $(X,G)$ is amenable if and only if the following holds: for every $\epsilon > 0$, every finite subset $E \subseteq G$, and every finite partition $\mathcal{U}$ of $X$ into clopen subsets, there exists a finite non-empty subset $F \subseteq X$ such that \begin{displaymath}
	\forall U \in \mathcal{U} \, \forall g \in E \colon \, \left\lvert \lvert F \cap U \rvert - \lvert g(F) \cap U \rvert \right\rvert \leq \epsilon \lvert F \rvert .
\end{displaymath} \end{cor}

In the following we shall point out how the criterion established above behaves when passing to dense or generating subsets of the acting group. These results are straightforward and we record them for the convenience of the reader.

\begin{prop}\label{proposition:generating.subsets} Let $(X,G)$ be a dynamical system and let $S$ be a symmetric generating subset of $G$ containing the identity map. Then \begin{displaymath}
	\inf_{E \in \mathcal{F}(G)} \inf_{\mathcal{U} \in \mathcal{N}(X)} \sup_{F \in \mathcal{F}_{+}(X)} \inf_{g \in E} \frac{\mu (F,g(F),\mathcal{U})}{|F|} = 1 
\end{displaymath} if and only if \begin{displaymath}
	\inf_{E \in \mathcal{F}(S)} \inf_{\mathcal{U} \in \mathcal{N}(X)} \sup_{F \in \mathcal{F}_{+}(X)} \inf_{g \in E} \frac{\mu (F,g(F),\mathcal{U})}{|F|} = 1 .
\end{displaymath} \end{prop}

\begin{proof} ($\Longrightarrow $) This is obvious.

($\Longleftarrow $) Let $\theta_{0} \in [0,1)$, $E_{0} \in \mathcal{F}(G)$ and $\mathcal{U} \in \mathcal{N}(X)$. Since $E_{0}$ is finite, there exist a finite subset  $E \subseteq S$ as well as $n \in \mathbb{N}\setminus \{ 0 \}$ such that $E_{0} \subseteq E^{n}$. By Lemma~\ref{lemma:uniform.star.refinements}, there exists $\mathcal{V} \in \mathcal{N}(X)$ such that $\mathcal{U} \preceq \mathcal{V}^{\ast,n-1}$. Consider $\mathcal{W} \defeq \bigvee_{g \in E^{n}} g^{-1}(\mathcal{V})$ and $\theta \defeq \frac{n+\theta_{0}-1}{n} \in [0,1)$. By assumption, there exists $F \in \mathcal{F}_{+}(X)$ such that $\inf_{s \in E} \mu (F,s(F),\mathcal{W}) \geq \theta |F|$. We argue that $\inf_{g \in E_{0}} \mu(F,g(F),\mathcal{W}) \geq \theta_{0}|F|$. To this end, let $g \in E_{0}$. Then there exist $g_{1},\ldots,g_{n} \in E$ such that $g = g_{n}\cdots g_{1}$. For each $i \in \{ 1,\ldots, n \}$, let $s_{i} \defeq g_{n}\cdots g_{i}$. Note that \begin{displaymath}
	\mu (s_{i+1}(F),s_{i}(F),\mathcal{V}) = \mu (s_{i+1}(F),s_{i+1}(g_{i}(F)),\mathcal{V}) = \mu (F,g_{i}(F),s_{i+1}^{-1}(\mathcal{V}))
\end{displaymath} for each $i \in \{ 1,\ldots,n-1 \}$. Hence, \begin{align*}
	\mu (F,g(F),\mathcal{U}) &\stackrel{\ref{corollary:composition.of.matchings}}{\geq} \mu (F,g_{n}(F),\mathcal{V}) + \sum_{i=1}^{n-1} \mu (s_{i+1}(F),s_{i}(F),\mathcal{V}) - \sum_{i=1}^{n-1} |s_{i}(F)| \\
	&= \mu (F,g_{n}(F),\mathcal{V}) + \sum_{i=1}^{n-1} \mu (F,g_{i}(F),s_{i+1}^{-1}(\mathcal{V})) - (n-1) |F| \\
	&\geq \mu (F,g_{n}(F),\mathcal{W}) + \sum_{i=1}^{n-1} \mu (F,g_{i}(F),\mathcal{W}) - (n-1) |F| \\
	&\geq n \theta |F| - (n-1)|F| = \theta_{0}|F| .
\end{align*} Consequently, $\inf_{g \in E_{0}} \mu (F,g(F),\mathcal{U}) \geq \theta_{0}|F|$. This proves our claim. \end{proof}

\begin{prop}\label{proposition:dense.subgroups} Let $X$ be a uniform space and let $H \subseteq G \subseteq \Aut (X)$. If $H$ is dense in $G$ with respect to the topology of uniform convergence, then \begin{displaymath}
	\inf_{E \in \mathcal{F}(G)} \inf_{\mathcal{U} \in \mathcal{N}(X)} \sup_{F \in \mathcal{F}_{+}(X)} \inf_{g \in E} \frac{\mu (F,g(F),\mathcal{U})}{\vert F \vert} = \inf_{E \in \mathcal{F}(H)} \inf_{\mathcal{U} \in \mathcal{N}(X)} \sup_{F \in \mathcal{F}_{+}(X)} \inf_{g \in E} \frac{\mu (F,g(F),\mathcal{U})}{\vert F \vert} .
\end{displaymath} \end{prop}

\begin{proof} Evidently, the left-hand side of the desired equation is bounded from above by the right-hand side, which we denote by $\rho$. In order to prove the converse inequality, let $\epsilon > 0$, $E \in \mathcal{F}(G)$, and $\mathcal{U} \in \mathcal{N}(X)$. By Lemma~\ref{lemma:uniform.star.refinements}, there exists $\mathcal{V} \in \mathcal{N}(X)$ such that $\mathcal{U} \preceq^{\ast} \mathcal{V}$. Since $H$ is dense in $G$, for each $g \in E$ there exists some $\hat{g} \in H \cap [g,\mathcal{V}]$. Let $F \in \mathcal{F}_{+}(X)$ such that $\inf_{g \in E} \mu (F,\hat{g}(F),\mathcal{V}) \geq (\rho -\epsilon)|F|$. Since $\hat{g}(x) \in \St (g(x),\mathcal{V})$ for all $x \in F$ and $g \in E$, we conclude that $\mu (F,g(F),\mathcal{U}) \geq \mu (F,\hat{g}(F),\mathcal{V})$ for every $g \in G$. Hence, \begin{displaymath}
		\inf_{g \in E} \frac{\mu (F,g(F),\mathcal{U})}{|F|} \geq \inf_{g \in E} \frac{\mu (F,\hat{g}(F),\mathcal{V})}{|F|} \geq \rho - \epsilon .
\end{displaymath} This completes the proof. \end{proof}

\begin{cor}\label{corollary:finite.generating.subsets} Let $X$ be a perfect Hausdorff space and let $G$ be a subgroup of $\Aut (X)$. Suppose that $S$ is a finite symmetric subset of $G$ containing the identity map and assume that the group generated by $S$ is dense in $G$ with respect to the topology of uniform convergence. Then $(X,G)$ is amenable if and only if the following holds: for every $\theta \in [0,1)$ and every finite uniform covering $\mathcal{U}$ of $X$, there exists a finite non-empty subset $F\subseteq X$ such that \begin{displaymath}
	\forall g \in S \colon \,  \mu (F,g(F),\mathcal{U}) \geq \theta \vert F \vert .
\end{displaymath} \end{cor}

\section{Matchings in topological groups}\label{section:matchings.in.topological.groups}

In this section we have a closer look at topological matchings in topological groups. This will lead to several new characterizations of amenability for Hausdorff topological groups.

\begin{thm}\label{theorem:amenable.groups} A Hausdorff topological group $G$ is amenable if and only if the following holds: for every $\theta \in [0,1)$, every finite subset $E \subseteq G$, and every finite uniform covering $\mathcal{U}$ of $G_{r}$, there exists a finite non-empty subset $F \subseteq G$ such that \begin{displaymath}
	\forall g \in E \colon \, \mu (F,gF,\mathcal{U}) \geq \theta \vert F \vert .
\end{displaymath} \end{thm}

\begin{proof} ($\Longleftarrow $) This is due to Theorem~\ref{theorem:vanishing.matching.number.implies.amenability}.

($\Longrightarrow $) Suppose that $G$ is amenable. If $G$ is discrete, then Theorem~\ref{theorem:folner} asserts the following: for every $\theta \in [0,1)$ and every finite subset $E\subseteq G$, there exists a finite non-empty subset $F \subseteq G$ such that $\vert F \cap gF \vert \geq \vert F \vert$ for all $g \in E$, and thus \begin{displaymath}
	\mu (F,gF,\mathcal{U}) \stackrel{\ref{remark:estimate.for.matching.number}}{\geq} \vert F \cap gF \vert \geq \vert F \vert
\end{displaymath} for all $g \in E$ and every covering $\mathcal{U}$ of $G$. Otherwise, if $G$ is not discrete, then $G$ is perfect by Remark~\ref{remark:isolated.points}, and hence the conclusion follows from Corollary~\ref{corollary:amenable.perfect.hausdorff.dynamical.systems}. \end{proof}

\begin{cor}\label{corollary:generators.of.dense.subgroups} Let $G$ be a Hausdorff topological group and let $S \subseteq G$ be a symmetric subset containing the neutral element and generating a dense subgroup of $G$. Then $G$ is amenable if and only if the following holds: for every $\theta \in [0,1)$, every finite subset $E \subseteq S$, and every finite uniform covering $\mathcal{U}$ of $G_{r}$, there is a finite non-empty subset $F \subseteq G$ with \begin{displaymath}
		\forall g \in E \colon \, \mu (F,gF,\mathcal{U}) \geq \theta \vert F \vert .
	\end{displaymath} \end{cor}

\begin{proof} Denote by $H$ the subgroup of $G$ generated by $S$. Since the homomorphism $\lambda_{G} \colon G \to \Aut (G_{r})$ is continuous, the subgroup $\lambda_{G}(H)$ is dense in $\lambda_{G}(G)$ with respect to the topology of uniform convergence. Accordingly, \begin{align*}
	G \textnormal{ amenable } \ &\stackrel{\ref{theorem:amenable.groups}}{\Longleftrightarrow } \ \inf_{E \in \mathcal{F}(G)} \inf_{\mathcal{U} \in \mathcal{N}(G_{r})} \sup_{F \in \mathcal{F}_{+}(G)} \inf_{g \in E} \frac{\mu (F,gF,\mathcal{U})}{|F|} = 1 \\
	& \stackrel{\ref{proposition:dense.subgroups}}{\Longleftrightarrow } \ \inf_{E \in \mathcal{F}(H)} \inf_{\mathcal{U} \in \mathcal{N}(G_{r})} \sup_{F \in \mathcal{F}_{+}(G)} \inf_{g \in E} \frac{\mu (F,gF,\mathcal{U})}{|F|} = 1 \\
	&\stackrel{\ref{proposition:generating.subsets}}{\Longleftrightarrow } \ \inf_{E \in \mathcal{F}(S)} \inf_{\mathcal{U} \in \mathcal{N}(G_{r})} \sup_{F \in \mathcal{F}_{+}(G)} \inf_{g \in E} \frac{\mu (F,gF,\mathcal{U})}{|F|} = 1 .\qedhere
\end{align*} \end{proof}

So far so good. We are now coming to material that prepares the study of a relationship with continuous Ramsey theory in the last section. Our next objective is to significantly strengthen Theorem~\ref{theorem:amenable.groups}. As it turns out, amenability of topological groups can be characterized in terms of matching properties involving only two-element uniform coverings (see Theorem~\ref{theorem:two.element.coverings} and Corollary~\ref{corollary:two.element.coverings}). This observation generalizes a recent result by Moore \cite{moore} for discrete groups.

In order state and prove the desired Theorem~\ref{theorem:two.element.coverings}, we need to agree on some additional terminology. To this end, let $G$ be a topological group and let $f \colon G \to \mathbb{R}$. Given an identity neighborhood $U$ in $G$ and any $\epsilon > 0$, we say that $f$ is \emph{$(U,\epsilon)$-uniformly continuous} if $\diam f(Ux) \leq \epsilon$ for each $x \in G$. Note that $f$ is a uniformly continuous map from $G_{r}$ to $\mathbb{R}$ if and only if for every $\epsilon > 0$, there exists an identity neighborhood $U$ in $G$ such that $f$ is $(U,\epsilon)$-uniformly continuous. Of course, this concept refers to the right uniformity of $G$ and may be defined analogously with respect to the corresponding left uniformity. However, we shall only need it as given here.

Now, let us note the following observation. 

\begin{lem}\label{lemma:relaxed.functions} Let $G$  be a topological group and $U$ an identity neighborhood in $G$. Suppose that $f \in {\rm UC}_{b}(G_{r})$ and $\epsilon > 0$. If $f$ is $(U,\epsilon)$-uniformly continuous, then $f_{\nu}$ is $(U,\epsilon)$-uniformly continuous for every mean $\nu \in M(G_{r})$. \end{lem}

\begin{proof} First let us observe that \begin{align*}
	N &\defeq \{ \nu \in M(G_{r}) \mid f_{\nu} \text{ is $(U,\epsilon)$-uniformly continuous} \} \\
	&= \bigcap_{x \in G, \, y \in Ux} \{ \nu \in M(G_{r}) \mid \vert f_{\nu}(x) - f_{\nu}(y) \vert \leq \epsilon \}
\end{align*} is closed in $M(G_{r})$ with respect to the weak-* topology. We argue that $N$ is convex. To this end, let $\nu, \nu' \in N$ and $\alpha \in [0,1]$, and consider the mean $\mu \defeq \alpha \nu + (1-\alpha)\nu'$ on $G$. Note that \begin{displaymath}
	f_{\mu}(x) = \mu (f \circ \lambda_{G}(x)) = \alpha \nu (f \circ \lambda_{G}(x)) + (1-\alpha)\nu'(f \circ \lambda_{G}(x)) = \alpha f_{\nu}(x) + (1-\alpha)f_{\nu'}(x)
\end{displaymath} for all $x \in G$. Therefore, \begin{displaymath}
	\vert f_{\mu}(x) - f_{\mu}(y) \vert \leq \alpha \vert f_{\nu}(x) - f_{\nu}(y) \vert + (1-\alpha )\vert f_{\nu'}(x) - f_{\nu'}(y) \vert \leq \epsilon
\end{displaymath} for all $x,y \in G$ with $xy^{-1} \in U$, which means that $\alpha \mu + (1-\alpha)\nu \in N$. This shows that $N$ is convex. Note that $f_{\nu_{g}}(x) = \nu_{g}(f \circ \lambda_{G}(x)) = f(xg)$ for all $g,x \in G$. Hence, if $g \in G$, then the implication \begin{displaymath}
	xy^{-1} \in U \ \Longrightarrow \ xg(yg)^{-1} \in U \ \Longrightarrow \ \vert f_{\nu_{g}}(x) - f_{\nu_{g}}(y) \vert = \vert f(xg) - f(yg) \vert \leq \epsilon
\end{displaymath} holds for all $x,y \in G$, and thus $f_{\nu_{g}}$ is $(U,\epsilon)$-uniformly continuous. Of course, this means that $\{ \nu_{g} \mid g \in G \} \subseteq N$. Consequently, $M(G_{r}) = N$ by Theorem~\ref{theorem:general.means}. \end{proof}

Now everything is prepared to state and prove the aforementioned result. The interesting part of the theorem is that conditions (2)-(4) need to hold just for a \emph{single} $k$ and it is not necessary to assume that they hold for \emph{all} $k$. Moreover, in conditions (4)-(6) it is interesting to note that $S$ can be chosen uniform, just depending on the set $E \subset G$ and the neighborhood $U$. We will heavily use these improvements in the later parts of the paper. The structure of the proof and also the main ideas go back to the original proof of Moore \cite{moore} in the discrete case -- however, some arguments might be somewhat streamlined and some others needed a careful adaption to the topological case.

Some generalization of Moore's work on discrete groups to the case of polish groups appeared in work of Ka{\"i}chouh \cite{kaichouh}, however our approach captures more aspects and also seems to be better suited for the Ramsey theoretic applications that one might have in mind. We obtain Ka{\"i}chouh's results as a straightforward corollary -- once we proved our main theorem.

\begin{thm}\label{theorem:two.element.coverings} Let $G$ be a Hausdorff topological group. For every natural number $k \geq 2$ and every $\theta \in (0,1)$, the following are equivalent: \begin{enumerate}
	\item[\emph{(1)}] $G$ is amenable.
	\item[\emph{(2)}] For every finite subset $E \subseteq G$ and every finite uniform covering $\mathcal{U}$ of $G_{r}$, there exists a finite non-empty subset $F \subseteq G$ such that \begin{displaymath}
		\forall g \in E \colon \, \mu (F,gF,\mathcal{U}) \geq \theta \vert F \vert .
	\end{displaymath}
	\item[\emph{(3)}] For every finite subset $E \subseteq G$ and every uniform covering $\mathcal{U}$ of $G_{r}$ with $\vert \mathcal{U} \vert \leq k$, there is a finite non-empty subset $F \subseteq G$ such that \begin{displaymath}
		\forall g,h \in E \colon \, \mu (gF,hF,\mathcal{U}) \geq \theta \vert F \vert .
	\end{displaymath}
	\item[\emph{(4)}] For every finite subset $E \subseteq G$ and every identity neighborhood $U$ in $G$, there is a finite subset $S \subseteq G$ such that, for every $U$-uniform covering $\mathcal{U}$ of $G_{r}$ with $\vert \mathcal{U} \vert \leq k$, there exists a non-empty subset $F \subseteq S$ such that \begin{displaymath}
		\forall g,h \in E \colon \, \mu (gF,hF,\mathcal{U}) \geq \theta \vert F \vert .
	\end{displaymath}
	\item[\emph{(5)}] There exists $\epsilon \in (0,1)$ such that for every finite subset $E \subseteq G$ and every identity neighborhood $U$ in $G$, there exists a finite subset $S \subseteq G$ such that for every $(U,\tfrac{1}{8})$-uniformly continuous $f \in {\rm UC}_{b}(G_{r})$ with $f(S) \subseteq [0,1]$, there is a non-empty subset $F \subseteq S$ such that $EF \subseteq S$ and \begin{displaymath}
		\forall g,h \in E \colon \, \vert \nu_{gF} (f) - \nu_{hF}(f) \vert \leq \epsilon .
	\end{displaymath}
	\item[\emph{(6)}] For every $\epsilon > 0$, every finite subset $E \subseteq G$, and every identity neighborhood $U$ in $G$, there exists a finite subset $S \subseteq G$ such that for every $(U,\tfrac{1}{8})$-uniformly continuous $f \in {\rm UC}_{b}(G_{r})$ with $f(S) \subseteq [0,1]$, there exists $\delta \in \Delta (G)$ with $(\spt \delta) \cup E(\spt \delta) \subseteq S$ such that \begin{displaymath}
		\forall g,h \in E \colon \, \vert (\nu_{g}\nu_{\delta})(f) - (\nu_{h}\nu_{\delta})(f) \vert \leq \epsilon .
	\end{displaymath}
\end{enumerate} \end{thm}
	
\begin{proof} From Theorem~\ref{theorem:amenable.groups}, we already know that (1) implies (2). The remaining part of the proof proceeds as follows: (2)$\Longrightarrow$(3)$\Longrightarrow$(4)$\Longrightarrow$(5)$\Longrightarrow$(1).
		
	(2)$\Longrightarrow$(3). Consider a finite subset $E \subseteq G$ as well as some finite uniform covering $\mathcal{U}$ of $G_{r}$. Then $\mathcal{V} \defeq \bigvee_{g \in E} g^{-1}(\mathcal{U})$ is a finite uniform covering of $G_{r}$. By (2), there exists a finite non-empty subset $F \subseteq G$ such that $\mu (F,gF,\mathcal{V}) \geq \theta \vert F \vert$ for every $g \in E^{-1}E$. Now, if $g,h \in E$, then $\mathcal{V}$ refines $g^{-1}(\mathcal{U})$, and thus \begin{displaymath}
		\mu (gF,hF,\mathcal{U}) = \mu (F,g^{-1}hF,g^{-1}(\mathcal{U})) \geq \mu (F,g^{-1}hF,g^{-1}(\mathcal{V})) \geq \theta \vert F \vert .
	\end{displaymath}
		
	(2)$\Longrightarrow$(3). Our proof proceeds by contradiction. So, assume that (3) does not hold for some finite subset $E \subseteq G$ and an identity neighborhood $U$ in $G$. That is, for every finite subset $S \subseteq G$, there exists a $U$-uniform covering $\mathcal{U}$ of $G$ such that $\vert \mathcal{U} \vert \leq k$ and \begin{displaymath}
		\forall F \subseteq S , \, F \ne \emptyset , \, \exists g,h \in E \colon \, \mu (gF,hF,\mathcal{U}) < \theta \vert F \vert .
	\end{displaymath} Equivalently, for each finite subset $S \subseteq G$, the set \begin{multline*}
		\Delta(S) \defeq \{ (V_{1},\ldots,V_{k}) \in \mathcal{P}(G)^{k} \mid \{ V_{1},\ldots,V_{k} \} \text{ $U$-uniform covering of } G_{r}, \\
		\forall F \subseteq S , \, F \ne \emptyset \, \exists g,h \in E \colon \, \mu (gF,hF,\{V_{1},\ldots,V_{k}\}) < \theta \vert F \vert \}
	\end{multline*} is non-empty. Of course, $\Omega (S) \defeq \{ ( V_{1} \cap S, \ldots, V_{k} \cap S ) \mid (V_{1},\ldots,V_{k}) \in \Delta (S) \}$ is finite for every finite subset $S \subseteq G$. Hence, the product space \begin{displaymath}
		X \defeq \prod_{S \in \mathcal{F}(G)} \Omega (S)
	\end{displaymath} with respect to the discrete spaces $\Omega(S)$ ($S \in \mathcal{F}(G)$) is compact. We show that the subset \begin{displaymath}
		Y \defeq \{ (V_{S,1},\ldots,V_{S,k})_{S \in \mathcal{F}(G)} \in X \mid \forall T \subseteq S \in \mathcal{F}(G) \, \forall i \in \{ 1,\ldots, k \} \colon \, V_{T,i} = V_{S,i} \cap T \}
	\end{displaymath} is not empty. Towards this aim, consider the closed subsets \begin{displaymath}
		Y_{S} \defeq \{ (V_{S,1},\ldots,V_{S,k})_{S \in \mathcal{F}(G)} \in X \mid \forall T \subseteq S \, \forall i \in \{ 1,\ldots, k \} \colon \, V_{T,i} = V_{S,i} \cap T \} \quad (S \in \mathcal{F}(G)) .
	\end{displaymath} We claim that $\mathcal{Y} \defeq \{ Y_{S} \mid S \in \mathcal{F}(G) \}$ has the finite intersection property. Note that for any finite sequence of finite subsets $S_{1},\ldots,S_{n} \subseteq G$, it is true that \begin{displaymath}
		Y_{\bigcup_{i = 1}^{n} S_{i}} \subseteq \bigcap_{i =1}^{n} Y_{S_{i}} .
	\end{displaymath} Hence, it suffices to show that $Y_{S} \ne \emptyset$ for every finite subset $S \subseteq G$. So, consider a finite subset $S \subseteq G$ and let $(V_{1},\ldots,V_{k}) \in \Delta (S)$. Then $(V_{1},\ldots,V_{k}) \in \Delta (T)$ for every subset $T \subseteq S$. Let us choose any element $(V_{T,1},\ldots,V_{T,k})_{T \in \mathcal{F}(G)\setminus \mathcal{P}(S)}$ of $\prod_{S \in \mathcal{F}(G)\setminus \mathcal{P}(S)} \Omega (S)$. We obtain an element $(V_{T,1},\ldots,V_{T,1})_{T \in \mathcal{F}(G)}$ of $Y_{S}$ by setting \begin{displaymath}
		(V_{T,1},\ldots, V_{T,k}) \defeq \begin{cases}
		(V_{1} \cap T,\ldots, V_{k} \cap T) & \text{if } T \subseteq S, \\
		(V_{T,1},\ldots,V_{T,k}) & \text{otherwise}
		\end{cases} \qquad (T \in \mathcal{F}(G)) .
	\end{displaymath} Thus, $Y_{S} \ne \emptyset$. By compactness of $X$, it follows that $Y = \bigcap \mathcal{Y}$ is non-empty. Consider any $(V_{S,1},\ldots,V_{S,k})_{S \in \mathcal{F}(G)} \in Y$. Let $V_{i} \defeq \bigcup \{ V_{S,i} \mid S \in \mathcal{F}(G) \}$ for each $i \in \{ 1,\ldots, k \}$. We claim that $\mathcal{U} \defeq \{ V_{1},\ldots,V_{k} \}$ is a $U$-uniform covering of $G$. To prove this, let $x \in G$. Suppose that $Ux \nsubseteq V_{i}$ for all $i \in \{ 1,\ldots, k\}$. For each $i \in \{ 1,\ldots,k \}$, choose an element $y_{i} \in Ux\setminus V_{i}$. Put $S \defeq \{ x,y_{1},\ldots,y_{k} \}$. Then there exists $(W_{1},\ldots,W_{k}) \in \mathcal{P}(G)^{k}$ such that $\{ W_{1},\ldots,W_{k} \}$ is a $U$-uniform covering of $G$ and $V_{S,i} = W_{i} \cap S$ for each $i \in \{ 1,\ldots, k\}$. Accordingly, there exists $i \in \{ 1,\ldots, k\}$ such that $Ux \subseteq W_{i}$ and thus $\{ x,y_{1},\ldots,y_{k} \} = S \subseteq V_{S,i} \subseteq V_{i}$. However, this clearly constitutes a contradiction. Consequently, $Ux \subseteq V_{i}$ for some $i \in \{ 1,\ldots, k\}$. This means that $\mathcal{U}$ is a $U$-uniform covering of $G$. We argue that \begin{displaymath}
		\forall F \in \mathcal{F}_{+}(G) \, \exists g,h \in E \colon \, \mu (gF,hF,\mathcal{U}) < \theta \vert F \vert ,
	\end{displaymath} which would clearly contradict (2) and hence prove that (2) implies (3). For this purpose, consider a finite, non-empty subset $F \subseteq G$. Let $S \defeq F \cup EF$. By construction, there is $(W_{1},\ldots,W_{k}) \in \Delta(S)$ such that $V_{S,i} = W_{i} \cap S$ for every $i \in \{ 1,\ldots, k\}$. That is, $V_{i} \cap S = W_{i} \cap S$ for all $i \in \{ 1,\ldots, k \}$. Hence, \begin{displaymath}
		\mu (gF,hF,\mathcal{U}) = \mu (gF,hF,\{ W_{1},\ldots,W_{k} \})
	\end{displaymath} for all $g,h \in E$. Since $F \subseteq S$ and $(W_{1},\ldots,W_{k}) \in \Delta(S)$, there exist $g,h \in E$ such that \begin{displaymath}
		\mu (gF,hF,\{ W_{1},\ldots,W_{k} \}) < \theta \vert F \vert .
	\end{displaymath} Accordingly, $\mu (gF,hF,\mathcal{U}) < \theta \vert F \vert$. This contradicts (2) and hence finishes the argument.
		
	(3)$\Longrightarrow$(4). We show the desired statement (4) for $\epsilon \defeq 1-\frac{\theta}{3}$. To this end, let $E \subseteq G$ be finite and let $U$ be an identity neighborhood in $G$. Suppose $S$ to be as in (3). We prove (4) for $S' \defeq S \cup ES$. Let $f \in {\rm UC}_{b}(G_{r})$ be $(U,\tfrac{1}{8})$-uniformly continuous with $f(S') \subseteq [0,1]$. Consider the two-element $U$-uniform covering $\mathcal{U} \defeq \{ V,W \}$ where $V \defeq f^{-1}((-\infty,\tfrac{2}{3}))$ and $W \defeq f^{-1}((\tfrac{1}{3},\infty ))$. According to (3), there exists a non-empty subset $F \subseteq S$ such that $\mu (gF,hF,\mathcal{U}) \geq \theta \vert F \vert$ for all $g,h \in E$. Let $g,h \in E$. Then there exists an injective map $\phi \colon D \to hF$ such that $D \subseteq gF$ and $|D| = \mu (gF,hF,\mathcal{U})$, as well as \begin{displaymath}
		\forall x \in D \colon \, \{ x,\phi(x)\} \subseteq V \vee \{ x,\phi(x)\} \subseteq W .
	\end{displaymath} Consider any bijection $\bar{\phi} \colon gF \to hF$ with $\bar{\phi}|_{D} = \phi|_{D}$. We conclude that \begin{align*}
		|\nu_{gF}(f) - \nu_{hF}(f)| &= \frac{1}{|F|}\left|\sum_{x \in gF} f(x) - \sum_{x \in hF} f(x)\right| \\
		&= \frac{1}{|F|}\left|\sum_{x \in D} (f(x)-f(\bar{\phi} (x))) + \sum_{x \in (gF)\setminus D} (f(x)-f(\bar{\phi} (x))) \right| \\
		&\leq \frac{1}{|F|}\left(\sum_{x \in D} |f(x)-f(\phi (x))| + \sum_{x \in (gF)\setminus D} |f(x) - f(\bar{\phi}(x))|\right) \\
		&\leq \frac{1}{|F|}\left(\frac{2}{3}\vert D \vert + \vert F \vert - \vert D \vert \right) = \frac{1}{|F|}\left( \vert F \vert - \frac{1}{3}\vert D \vert \right) \\
		&\leq 1-\frac{\theta}{3} = \epsilon .
	\end{align*}
		
	(4)$\Longrightarrow$(5). Suppose $\epsilon \in (0,1)$ to be as in (4). Let $\bar{\epsilon} > 0$. Consider a finite subset $E \subseteq G$ and an identity neighborhood $U$ in $G$. Let $n \in \mathbb{N}$ such that $\epsilon^{n} < \bar{\epsilon}$. First, for each $i \in \{ 1,\ldots ,n \}$, let $U_{i}$ be an identity neighborhood in $G$ such that $f$ is $(U_{i},\tfrac{1}{8}\epsilon^{n-i-1})$-uniformly continuous. Second, we recursively choose a sequence of finite subsets $E_{0},\ldots,E_{n} \subseteq G$. Let $E_{0} \defeq E$. For $i \in \{ 1,\ldots ,n \}$, choose a finite subset $E_{i} \subseteq G$ satisfying (4) with respect to $E_{i-1}$ and $U_{i-1}$. Put $S \defeq E_{n}$. By downward recursion, we construct non-empty subsets $F_{i} \subseteq E_{i+1}$ ($i \in \{ 0,\ldots, n-1 \}$) such that $E_{i}F_{i} \subseteq E_{i+1}$ and \begin{displaymath}
		\forall g,h \in E_{i} \colon \, \vert (\nu_{g}\nu_{F_{i}}\cdots \nu_{F_{n-1}})(f) - (\nu_{h}\nu_{F_{i}}\cdots \nu_{F_{n-1}})(f) \vert \leq \epsilon^{n-i} .
	\end{displaymath} For a start, let $f_{n-1} \defeq f$ and choose a non-empty subset $F_{n -1} \subseteq E_{n}$ with $E_{n-1}F_{n-1} \subseteq E_{n}$ and \begin{displaymath}
		\forall g,h \in E_{n-1} \colon \, \vert \nu_{gF_{n-1}}(f_{n-1}) - \nu_{hF_{n-1}}(f_{n-1}) \vert \leq \epsilon .
	\end{displaymath} For the recursion, suppose that non-empty subsets $F_{n-1} \subseteq E_{n}, \ldots ,F_{i+1} \subseteq E_{i+2}$ and functions $f_{n-1},\ldots,f_{i+1} \in {\rm UC}_{b}(G_{r})$ have been chosen such that $E_{j}F_{j} \subseteq E_{j+1}$, $f_{j}$ is $(U_{j},\delta_{j})$-uniformly continuous, and \begin{displaymath}
		\forall j \in \{ i+1,\ldots,n-1 \} \, \forall g,h \in E_{j} \colon \, \vert \nu_{gF_{j}}(f_{j}) - \nu_{hF_{j}}(f_{j}) \vert \leq \epsilon .
	\end{displaymath} Define a bounded uniformly continuous function $f_{i} \colon G_{r} \to \mathbb{R}$ by \begin{align*}
		f_{i}(g) &\defeq \epsilon^{i+1-n}\left( f_{\nu_{F_{i+1}}\cdots \nu_{F_{n-1}}}(g) - \min_{h \in E_{i+1}}f_{\nu_{F_{i+1}}\cdots \nu_{F_{n-1}}}(h) \right) \\
		&= \epsilon^{i+1-n}\left( (\nu_{g}\nu_{F_{i+1}}\cdots \nu_{F_{n-1}})(f) - \min_{h \in E_{i+1}} (\nu_{h}\nu_{F_{i+1}}\cdots \nu_{F_{n-1}})(f) \right) .
	\end{align*} Since $f_{\nu_{F_{i+1}}\cdots \nu_{F_{n-1}}}$ is $(U_{i},\tfrac{1}{8}\epsilon^{n-i-1})$-relaxed by Lemma~\ref{lemma:relaxed.functions}, we conclude that $f$ is $(U_{i},\tfrac{1}{8})$-relaxed. By induction hypothesis, it is furthermore true that $f_{i}(E_{i+1}) \subseteq [0,1]$. Hence, there exists a non-empty subset $F_{i} \subseteq E_{i+1}$ such that $E_{i}F_{i} \subseteq E_{i+1}$ and $\vert \nu_{gF_{i}}(f_{i}) - \nu_{hF_{i}}(f_{i}) \vert \leq \epsilon$ for all $g,h \in E_{i}$. Note that \begin{displaymath}
			\nu_{gF_{i}}\left(f_{\nu_{F_{i+1}}\cdots \nu_{F_{n-1}}} \right) = (\nu_{gF_{i}}\cdots \nu_{F_{n-1}})(f) = (\nu_{g}\nu_{F_{i}}\cdots \nu_{F_{n-1}})(f)
	\end{displaymath} for every $g \in G$. We conclude that \begin{align*}
		\theta^{i+1-n} \vert (\nu_{g}\nu_{F_{i}}\cdots \nu_{n-1})(f) - (\nu_{h}\nu_{F_{i}}\cdots \nu_{n-1})(f) \vert = \vert \nu_{gF_{i}}(f_{i}) - \nu_{hF_{i}}(f_{i}) \vert \leq \epsilon
	\end{align*} and thus $\vert (\nu_{g}\nu_{F_{i}}\cdots \nu_{F_{n-1}})(f) - (\nu_{h}\nu_{F_{i}}\cdots \nu_{F_{n-1}})(f) \vert \leq \epsilon^{n-i}$ for all $g,h \in E_{i}$. This completes the recursion. Finally, let $\delta \defeq \delta_{F_{0}}\cdots \delta_{F_{n-1}} \in \Delta(G)$. Note that $\spt (\delta ) \subseteq F_{0}\cdots F_{n-1} \eqdef F$. Evidently, $F \subseteq S$ and \begin{displaymath}
			EF = E_{0}F_{0}F_{1}\cdots F_{n-1} \subseteq E_{1}F_{1}F_{2}\cdots F_{n-1} \subseteq \ldots \subseteq E_{n-1}F_{n-1} \subseteq E_{n} = S .
	\end{displaymath} Besides, $\nu_{\delta} = \nu_{F_{0}}\cdots \nu_{F_{n-1}}$ and hence $\vert (\nu_{g}\nu_{\delta}) (f) - (\nu_{h}\nu_{\delta})(f) \vert \leq \epsilon^{n} < \bar{\epsilon}$ for all $g,h \in E_{0} = E$.
		
	(5)$\Longrightarrow$(1). By a compactness argument similar to the one in the proof of Theorem~\ref{theorem:vanishing.matching.number.implies.amenability}, it suffices to show that for every $\epsilon > 0$, every finite subset $E \subseteq G$, and every finite sequence of uniformly continuous functions $f_{1},\ldots,f_{n} \colon G_{r} \to [0,1]$, there exists $\nu \in M(G_{r})$ such that \begin{displaymath}
		\forall i \in \{ 1,\ldots, n \} \, \forall g \in E \colon \, \vert \nu(f_{i}) - \nu(f_{i} \circ \lambda_{G}(g)) \vert \leq \epsilon .
	\end{displaymath} Therefore, let $\epsilon > 0$ and consider a finite subset $E \subseteq G$ as well as a finite sequence of uniformly continuous functions $f_{1},\ldots,f_{n} \colon G_{r} \to [0,1]$. Suppose $U$ to be an identity neighborhood in $G$ such that $f_{i}$ is $(U,\tfrac{1}{8})$-relaxed for each $i \in \{ 1,\ldots, n \}$. We construct a sequence of finite subsets $E_{0},\ldots,E_{n} \subseteq G$ by recursion. Let $E_{0} \defeq E \cup \{ e \}$. For each $i \in \{ 1,\ldots, n\}$, choose a finite subset $E_{i} \subseteq G$ satisfying (5) with regard to $U$ and with $E_{i-1}$ in place of $E$ and $\tfrac{\epsilon}{2}$ in place of $\epsilon$. Next, we construct a sequence $\delta_{1},\ldots,\delta_{n} \in \Delta (G)$ by downward recursion. For a start, choose $\delta_{n} \in \Delta (G)$ such that $(\spt \delta_{n}) \cup E_{n-1}(\spt \delta_{n}) \subseteq E_{n}$ and \begin{displaymath}
		\forall g \in E_{n-1} \colon \, \vert \nu_{\delta_{n}}(f_{n}) - (\nu_{g}\nu_{\delta_{n}})(f_{n}) \vert \leq \tfrac{\epsilon}{2} .
	\end{displaymath} For the recursive step, suppose that $\delta_{n},\ldots,\delta_{i+1} \in \Delta (G)$ have been chosen such that, for each $j \in \{ i+1,\ldots,n \}$, we have $(\spt \delta_{j}) \cup E_{j-1}(\spt \delta_{j}) \subseteq E_{j}$ and \begin{displaymath}
		\forall g \in E_{j-1} \colon \, \vert (\nu_{\delta_{j}}\cdots \nu_{\delta_{n}})(f_{j}) - (\nu_{g}\nu_{\delta_{j}}\cdots \nu_{\delta_{n}})(f_{j}) \vert \leq \tfrac{\epsilon}{2}
	\end{displaymath} Note that $(f_{i})_{\nu_{\delta_{i+1}} \cdots \nu_{\delta_{n}}} \colon G_{r} \to [0,1]$ is uniformly continuous and $(U,\tfrac{1}{8})$-relaxed by Lemma~\ref{lemma:relaxed.functions}. According to (5), there hence exists $\delta_{i} \in \Delta (G)$ such that $(\spt \delta_{i}) \cup E_{i-1}(\spt \delta_{i}) \subseteq E_{i}$ and \begin{displaymath}
		\vert (\nu_{\delta_{i}}\cdots \nu_{\delta_{n}})(f_{i}) - (\nu_{g}\nu_{\delta_{i}}\cdots \nu_{\delta_{n}})(f_{i}) \vert = \vert \nu_{\delta_{i}}((f_{i})_{\nu_{\delta_{i+1}} \cdots \nu_{\delta_{n}}}) - (\nu_{g}\nu_{\delta_{i}})((f_{i})_{\nu_{\delta_{i+1}} \cdots \nu_{\delta_{n}}}) \vert \leq \tfrac{\epsilon}{2}
	\end{displaymath} for every $g \in E_{i-1}$. This completes the recursion. Finally, consider $\delta = \delta_{1}\cdots \delta_{n} \in \Delta (G)$. Let $i \in \{ 1,\ldots, n \}$. As \begin{displaymath}
		\vert \nu_{\delta_{i}}((f_{i})_{\nu_{\delta_{i+1}}\cdots \nu_{\delta_{n}}}) - \nu_{g}((f_{i})_{\nu_{\delta_{i}}\cdots \nu_{\delta_{n}}}) \vert = \vert \nu_{\delta_{i}}((f_{i})_{\nu_{\delta_{i+1}} \cdots \nu_{\delta_{n}}}) - (\nu_{g}\nu_{\delta_{i}})((f_{i})_{\nu_{\delta_{i+1}} \cdots \nu_{\delta_{n}}}) \vert \leq \tfrac{\epsilon}{2}
	\end{displaymath} for every $g \in E_{i-1}$, we conclude that \begin{displaymath}
		\vert \nu_{\delta_{i} \cdots \delta_{n}}(f_{i}) - (\nu_{\gamma \delta_{i} \cdots \delta_{n}}(f_{i}) \vert = \vert \nu_{\delta_{i}}((f_{i})_{\nu_{\delta_{i+1}}\cdots \nu_{\delta_{n}}}) - \nu_{\gamma}((f_{i})_{\nu_{\delta_{i}}\cdots \nu_{\delta_{n}}}) \vert \leq \tfrac{\epsilon}{2}
	\end{displaymath} for every $\gamma \in \Delta (G)$ with $\spt \gamma \subseteq E_{i-1}$. Now, let $g \in E$. We observe that \begin{align*}
		&\spt (\delta_{1}\cdots \delta_{i-1}) = (\spt \delta_{1}) \cdots (\spt \delta_{i-1}) = E_{1}(\spt \delta_{2}) \cdots (\spt \delta_{i-1}) \subseteq \ldots \subseteq E_{i-1} , \\
		&\spt (\delta_{g}\delta_{1}\cdots \delta_{i-1}) = g(\spt \delta_{1}) \cdots (\spt \delta_{i-1}) = E_{0}(\spt \delta_{1}) \cdots (\spt \delta_{i-1}) \subseteq \ldots \subseteq E_{i-1} .
	\end{align*} Therefore, it follows that $\vert \nu_{\delta}(f_{i}) - (\nu_{\delta_{i}\cdots \delta_{n}})(f_{i}) \vert = \vert (\nu_{\delta_{1}\cdots \delta_{i-1}}\nu_{\delta_{i}\cdots \delta_{n}})(f_{i}) - (\nu_{\delta_{i}\cdots \delta_{n}})(f_{i}) \vert \leq \tfrac{\epsilon}{2}$ and $\vert (\nu_{g}\nu_{\delta})(f_{i}) - (\nu_{\delta_{i}\cdots \delta_{n}})(f_{i}) \vert = \vert (\nu_{\delta_{g}\delta_{1}\cdots \delta_{i-1}}\nu_{\delta_{i}\cdots \delta_{n}})(f_{i}) - (\nu_{\delta_{i}\cdots \delta_{n}})(f_{i}) \vert \leq \tfrac{\epsilon}{2}$. Consequently, $\vert \nu_{\delta}(f_{i}) - (\nu_{g}\nu_{\delta})(f_{i}) \vert \leq \epsilon$. This finishes the proof. \end{proof}

We will now discuss some aspects and some more or less direct consequences of the previous theorem.
First of all, the previous theorem does not cover infinite uniform coverings, even though this might seem to be just as natural. The corresponding results are nevertheless true, but require a different setup. We will address this issue in a forthcoming paper.

A slightly cumbersome fact is still that in condition (2) in the previous theorem, we had to use a symmetric form of the matching condition (which might be a priori stronger than its asymmetric counterpart). Indeed, we are currently not able to remove this stronger requirement without going to covers of at least three elements. This is formulated in the next corollary.
	
\begin{cor}\label{corollary:two.element.coverings} Let $G$ be a Hausdorff topological group. The following are equivalent: \begin{enumerate}
	\item[\emph{(1)}] $G$ is amenable.
	\item[\emph{(2)}] There exists $\theta \in (\tfrac{1}{2},1]$ such that for every finite subset $E \subseteq G$ and every uniform covering $\mathcal{U}$ of $G_{r}$ with $\vert \mathcal{U} \vert \leq 3$, there is a finite non-empty subset $F \subseteq G$ such that \begin{displaymath}
		\forall g \in E \colon \, \mu (F,gF,\mathcal{U}) \geq \theta \vert F \vert .
	\end{displaymath}
	\item[\emph{(3)}] There exists $\theta \in (0,1]$ such that every finite subset $E \subseteq G$, and every two-element uniform covering $\mathcal{U}$ of $G_{r}$, there exists a finite non-empty subset $F \subseteq G$ such that \begin{displaymath}
		\forall g,h \in E \colon \, \mu (gF,hF,\mathcal{U}) \geq \theta \vert F \vert .
	\end{displaymath}
\end{enumerate} \end{cor}

\begin{proof} Note that Theorem~\ref{theorem:two.element.coverings} asserts that (3)$\Longleftrightarrow$(1)$\Longrightarrow$(2). Hence, we are left to prove that (2)$\Longrightarrow$(3). For this purpose, suppose that (2) holds for some $\theta_{0} \in (\frac{1}{2},1]$. We show (3) for $\theta_{1} \defeq 2\theta_{0} - 1$. To this end, let $E \subseteq G$ be finite and let $\mathcal{U} = \{ U_{0},U_{1} \}$ be a two-element uniform covering of $G$. By Lemma~\ref{lemma:uniform.covers}, there exists a uniformly continuous function $f \colon G_{r} \to [0,1]$ such that $\spt (f) \subseteq U_{0}$ and $\spt (1-f) \subseteq U_{1}$. Since $f$ is uniformly continuous, it follows that $\mathcal{V} \defeq \{ f^{-1}([0,\tfrac{1}{2})), \, f^{-1}((\tfrac{1}{4},\tfrac{3}{4})), \, f^{-1}((\tfrac{1}{2},1]) \}$ is a uniform covering of $G$. Of course, $\vert \mathcal{V} \vert \leq 3$. According to (2), there exists a finite non-empty subset $F \subseteq G$ such that $\mu (F,gF,\mathcal{V}) \geq \theta_{0} \vert F \vert$ for every $g \in E$. Using Lemma~\ref{lemma:composition.of.matchings}, we conclude that \begin{displaymath}
	\mu (gF,hF,\mathcal{V}^{\ast}) \geq \mu (gF,F,\mathcal{V}) + \mu (F,hF,\mathcal{V}) - \vert F \vert \geq (2\theta_{0} - 1) \vert F \vert = \theta_{1}\vert F \vert
\end{displaymath} for any two $g,h \in F$. This completes the proof. \end{proof}

As mentioned before, we obtain proofs of results from  \cite{kaichouh}. The following result was proven for Polish groups in \cite{kaichouh}.

\begin{cor}\label{corollary:measuring.single.functions} Let $G$ be a Hausdorff topological group. Then $G$ is amenable if and only if the following holds: for every finite subset $E \subseteq G$ and every uniformly continuous map $f \colon G_{r} \to [0,1]$, there exists a mean $\nu \in M(G_{r})$ such that $\nu(f \circ \lambda_{G}(g)) = \nu (f)$ for all $g \in E$. \end{cor}

\begin{proof} Evidently, the forward implication is valid. In order to prove the backward implication, we shall utilize the equivalence of (1) and (3) in Theorem~\ref{theorem:two.element.coverings}. Indeed we are going to establish (3) in Theorem~\ref{theorem:two.element.coverings} for $\theta = \frac{1}{2}$ and $k = 2$. So, consider a finite subset $E \subseteq G$ and a two-element uniform covering $\mathcal{U} = \{ U_{0}, U_{1} \}$ of $G_{r}$. By Lemma~\ref{lemma:uniform.covers}, there exists a uniformly continuous function $f \colon G_{r} \to [0,1]$ such that $\spt (f) \subseteq U_{0}$ and $\spt (1-f) \subseteq U_{1}$. Due to our assumption and Lemma~\ref{lemma:dense.set.of.means}, there is a finite non-empty subset $F \subseteq G$ such that \begin{displaymath}
	\left\vert \sum_{x \in F} f(gx) - \sum_{x \in F} f(hx) \right\vert \leq \frac{\vert F \vert}{4}
\end{displaymath} for all $g,h \in E$. Analogously to the argument in the proof of Theorem~\ref{theorem:amenability.implies.vanishing.matching.number}, it follows that $\mu (gF,hF,\mathcal{U}) \geq \frac{1}{2}\vert F \vert$ for all $g,h \in E$. This substantiates that (3) in Theorem~\ref{theorem:two.element.coverings} is satisfied for $\theta = \frac{1}{2}$ and $k = 2$. Hence, $G$ is amenable by Theorem~\ref{theorem:two.element.coverings}. \end{proof}

\section{Coset colorings of non-archimedean groups}\label{section:colorings.of.non.archimedean.groups}

In this brief section we reformulate the results of the previous one for non-archimedean groups in terms of coset colorings. Recall that a topological group $G$ is \emph{non-archimedean} if every neighborhood of the neutral element in $G$ contains an open subgroup of $G$. Clearly, the class of non-archimedean groups encompasses all discrete groups as well as all topological subgroups of $S_{\infty}$. It is also well known that any totally disconnected, locally compact Hausdorff topological group is non-archimedean \cite{dantzig}. Now, let $G$ be an arbitrary non-archimedean topological group. Then the right uniformity of $G$ is given by \begin{displaymath}
\{ \mathcal{U} \subseteq \mathcal{P}(G) \mid \exists H \textnormal{ open subgroup of } G \colon \, \mathcal{U} \preceq H\backslash G \} ,
\end{displaymath} where $H\backslash G \defeq \{ Hx \mid x \in G\}$ for $H \leq G$. In particular, for every finite uniform covering $\mathcal{U}$ of $G_{r}$ there exists an open subgroup $H$ of $G$ as well as a map $\phi \colon H\backslash G \to \{ 0,\ldots,n \}$ with $n \geq 1$ such that $\mathcal{U}$ is refined by $\{ (\phi \circ \pi_{H})^{-1}(i) \mid i \in \{ 0,\ldots,n \} \}$ where $\pi_{H} \colon G \to H\backslash G, \, x \mapsto Hx$. Therefore, as a consequence of Theorem~\ref{theorem:amenable.groups} and Theorem~\ref{theorem:two.element.coverings}, we immediately obtain the subsequent characterization of amenability for non-archimedean Hausdorff topological groups in terms of right coset colorings.

\begin{cor}\label{coro:totally.disconnected}
	If $G$ is a non-archimedean Hausdorff topological group, then the following statements are equivalent.
	\begin{enumerate}
		\item[$(1)$]	$G$ is amenable.
		\item[$(2)$]	For every $\epsilon > 0$, every open subgroup $H \leq G$, every map $\varphi \colon H\backslash G \to \{ 0,\ldots, n \}$ with $n \geq 1$, and every finite subset $E \subseteq G$, there is a finite non-empty subset $F \subseteq G$ with \begin{equation*}
			 \qquad \qquad \forall i \in \{ 0,\ldots,n \} \, \forall g \in E \colon \, \left\lvert \lvert F \cap (\varphi \circ \pi_{H})^{-1}(i) \rvert - \lvert gF \cap (\varphi \circ \pi_{H})^{-1}(i) \rvert \right\rvert \leq \varepsilon \lvert F \rvert . \end{equation*}
		\item[$(3)$]	For every $\epsilon > 0$, every open subgroup $H \leq G$, every subset $A \subseteq G$, and every finite subset $E \subseteq G$, there is a finite non-empty subset $F \subseteq G$ such that \begin{equation*}
			\qquad \quad \forall g \in E \colon \, \left\lvert \lvert F \cap HA \rvert - \lvert gF \cap HA \rvert \right\rvert \leq \varepsilon \lvert F \rvert . \end{equation*} \end{enumerate}
\end{cor}

Even in the case of discrete groups, this provides us with the following interesting characterization of amenability, which is due to Moore \cite{moore}.

\begin{cor}[\cite{moore}]\label{coro.discrete}
	Let $G$ be a discrete group. The following are equivalent.
	\begin{enumerate}
		\item[$(1)$]	$G$ is amenable.
		\item[$(2)$]	For every $\epsilon > 0$, every map $\varphi \colon G \to \{ 0,\ldots, n \}$ with $n \geq 1$, and every finite subset $E \subseteq G$, there exists a finite non-empty subset $F \subseteq G$ such that \begin{equation*}
			\qquad \qquad \forall i \in \{ 0,\ldots,n \} \, \forall g \in E \colon \, \left\lvert \lvert F \cap \varphi^{-1}(i) \rvert - \lvert gF \cap \varphi^{-1}(i) \rvert \right\rvert \leq \varepsilon \lvert F \rvert . \end{equation*}
		\item[$(3)$]	For every $\epsilon > 0$, every subset $A \subseteq G$, and every finite subset $E \subseteq G$, there exists a finite non-empty subset $F \subseteq G$ such that \begin{equation*}
			\qquad \quad \forall g \in E \colon \, \left\lvert \lvert F \cap A \rvert - \lvert gF \cap A \rvert \right\rvert \leq \varepsilon \lvert F \rvert . \end{equation*}\end{enumerate}
\end{cor}

We conclude this section with an application. It was shown by Giordano and de la Harpe \cite{GiordanoDeLaHarpe} that a countable discrete group is amenable if and only if every of its continuous actions on a Cantor space is amenable. Utilizing Theorem~\ref{theorem:two.element.coverings}, we shall establish an amenability criterion for general non-archimedean groups in terms of minimal subflows of certain canonical shift operations, which in particular provides a generalization of the mentioned result by Giordano and de la Harpe to second-countable non-archimedean Hausdorff topological groups. For a non-archimedean Hausdorff topological group $G$ and any open subgroup $H \leq G$, let us consider the $G$-flow $\alpha_{H} \colon G \to \Aut (2^{H\backslash G})$ given by \begin{displaymath}
	\alpha_{H}(g)(\phi)(Hx) \defeq \phi (Hxg) \qquad (g \in G, \, \phi \in 2^{H\backslash G}) .
\end{displaymath}

\begin{cor}\label{corollary:non.archimedean} A non-archimedean Hausdorff topological group $G$ is amenable if and only if, for every open subgroup $H \leq G$, every minimal subflow of $\alpha_{H} \colon G \to \Aut (2^{H\backslash G})$ is amenable. \end{cor}

\begin{proof} The forward implication is trivial, since amenability of $G$ implies amenability of any $G$-flow. Conversely, suppose that, for every open subgroup $H \leq G$, every minimal sub-action of $\alpha_{H} \colon G \to \Aut (2^{H\backslash G})$ is amenable. In order to prove that $G$ is amenable, we show that $G$ satisfies condition (3) in Corollary~\ref{coro:totally.disconnected}. To this end, let $H$ be an open subgroup of $G$. Consider the compact Hausdorff space $X \defeq 2^{H\backslash G}$. We abbreviate $g.\phi \defeq \alpha_{H}(g)(\phi)$ if $g \in G$ and $\phi \in X$. Of course, $U \defeq \{ \phi \in X \mid \phi(H) = 0 \}$ is a clopen subset of $X$. Consider a finite subset $E \subseteq G$, a mapping $\phi \colon H\backslash G \to 2$ and some $\epsilon > 0$. Denote by $Z$ the closure of $Y \defeq \{ g.\phi \mid g \in G \}$ in $X$. Clearly, $Z$ is an $\alpha_{H}$-invariant non-empty subset of $X$. Since $Z$ is compact, a standard application of Zorn's lemma asserts the existence of a minimal closed, $\alpha_{H}$-invariant, non-empty subset $C \subseteq Z$. By assumption, the subflow $G \to \Aut (C), \, g \mapsto \alpha_{H}(g)|_{C}$ is amenable. 
	
\textit{Claim.} There exists a finite non-empty subset $F_{0} \subseteq C$ so that $\vert \vert F_{0} \cap U \vert - \vert g.F_{0} \cap U \vert \vert \leq \epsilon \vert F_{0} \vert$ for all $g \in E$.

\textit{Proof.} Denote by $B$ the set of all isolated points of the space $C$. Note that $B$ is an open $\alpha_{H}$-invariant subset of $B$. Hence, either $B = \emptyset$ or $B = C$ due to minimality of $C$. If $B = \emptyset$, then $C$ is perfect and hence the claim follows by Corollary~\ref{corollary:amenable.cantor.dynamical.systems}. Otherwise, if $B = C$, then $C$ is discrete and thus finite, which readily implies our claim for $F_{0} = C$. \qed

Let $F_{0}$ be a finite non-empty subset of $C$ such that $\vert \vert F_{0} \cap U \vert - \vert g.F_{0} \cap U \vert \vert \leq \epsilon \vert F_{0} \vert$ for all $g \in E$. Since $Y$ is dense in the Hausdorff space $Z$, there exists an injective map $\gamma \colon F_{0} \to G$ such that $\psi(Hg) = (\gamma(\psi).\phi)(Hg)$ for all $g \in E \cup \{ e \}$ and $\psi \in F_{0}$. Let $F_{1} \defeq \gamma (F_{0})$ and $V \defeq (\phi \circ \pi_{H})^{-1}(0)$. For all $x \in F_{1}$ and $g \in E \cup \{ e \}$, \begin{displaymath}
	(\phi \circ \pi_{H})(gx) = \phi (Hgx) = (x.\phi)(Hg) = \gamma^{-1}(x)(Hg) = (g.\gamma^{-1}(x))(H)
\end{displaymath} and hence \begin{displaymath}
	gx \in V \ \Longleftrightarrow \ (\phi \circ \pi_{H})(gx) = 0 \ \Longleftrightarrow \ (g.\gamma^{-1}(x))(H) = 0 \ \Longleftrightarrow \ g.\gamma^{-1}(x) \in U .
\end{displaymath} Accordingly, $\vert gF_{1} \cap V \vert = \vert g.F_{0} \cap U \vert$ for every $g \in G \cup \{ e \}$, and thus $\vert \vert F_{1} \cap U \vert - \vert gF_{1} \cap U \vert \vert \leq \epsilon \vert F_{1} \vert$ for every $g \in E$. Consequently, condition (3) in Corollary~\ref{coro:totally.disconnected} is satisfied. Therefore, $G$ is amenable, and we are done. \end{proof}

\begin{cor} A second-countable non-archimedean Hausdorff topological group $G$ is amenable if and only if every minimal $G$-flow on a Cantor space is amenable. \end{cor}

\section{Perfect matching conditions}\label{section:strong.matching.conditions}

Since compact and extremely amenable topological groups are amenable, they satisfy any of the matching conditions investigated in Theorem~\ref{theorem:amenable.groups} and Theorem~\ref{theorem:two.element.coverings}. As one might expect, those particular topological groups have certain strong matching properties, which we discuss in this final section.

We begin with compact topological groups. In fact, they satisfy the following perfect matching condition with regard to arbitrary uniform coverings.

\begin{prop}\label{proposition:perfect.matchings.in.compact.groups} Let $G$ be a compact topological group. If $\mathcal{U}$ is a uniform covering of $G_{r}$, then there exists $F \in \mathcal{F}_{+}(G)$ such that $\mu (F,gF,\mathcal{U}) = |F|$ for all $g \in G$. \end{prop} 

\begin{proof} Let $\mathcal{U}$ be a uniform covering of $G_{r}$. Then there exists an open neighborhood $U$ of the neutral element in $G$ such that $\mathcal{U} \preceq \{ U^{-1}Ux \mid x \in G \}$. Since $G$ is compact, $V \defeq \bigcap_{g \in G} g^{-1}Ug$ is an open neighborhood of the neutral element in $G$. Besides, $gV = Vg$ for all $g \in G$. Let $F \in \mathcal{F}_{+}(G)$ such that $G = VF$ and $\inf \{ |E| \mid E \in \mathcal{F}_{+}(G), \, G = VE \} = |F|$. Let $g \in G$. Note that $VgF = gVF = G$. We consider the bipartite graph $\mathcal{B} \defeq (F,gF,R)$ where $R \defeq \{ (x,y) \in F \times gF \mid Vx \cap Vy \ne \emptyset \}$. Let $S \subseteq F$ and $T \defeq N_{\mathcal{B}}(S)$. We show that $|S| \leq |T|$. To this end, let $E \defeq (F\setminus S) \cup T$. We argue that $G = VE$. Clearly, if $z \in V(F\setminus S)$, then $z \in VE$. Otherwise, there exist $x \in S$ and $y \in gF$ such that $z \in Vx \cap Vy$, which readily implies that $y \in N_{\mathcal{B}}(S)$ and thus $z \in VT \subseteq VE$. Therefore, $G = VE$. Accordingly, $|F| \leq |E|$ and hence $|S| \leq |T|$. Consequently, Corollary~\ref{corollary:hall} asserts that $\mathcal{B}$ admits a perfect matching. That is, $\mu(\mathcal{B}) = |F|$ and thus $\mu (F,gF,\mathcal{U}) = |F|$. \end{proof}

The previous observation immediately implies the subsequent result on approximately compact groups. We call a topological group $G$ \emph{compactly approximable} if there exists a directed set of compact subgroups of $G$ whose union is dense in $G$. Note that our concept of compact approximability is slightly more general than the one considered in \cite{KechrisRosendal}, since we do not require the directed set of compact subgroups to be countable.

\begin{prop}\label{proposition:compactly.approximable} Let $G$ be a compactly approximable topological group. If $E \subseteq G$ is finite and $\mathcal{U}$ is a uniform covering of $G_{r}$, then there exists $F \in \mathcal{F}_{+}(G)$ such that $\mu (F,gF,\mathcal{U}) = |F|$ for all $g \in E$. \end{prop} 

\begin{proof} Let $\mathcal{U}$ be a uniform covering of $G_{r}$. Then there exists an open neighborhood $V$ of the neutral element in $G$ such that $\mathcal{U} \preceq \{ VVV^{-1}x \mid x \in G \}$. Since $G$ is compactly approximable, there is a compact subgroup $H$ of $G$ with $E \subseteq VH$. Now, $\mathcal{V} \defeq \{ (V \cap H)x \mid x \in H \}$ is a uniform covering of $H_{r}$. Due to Proposition~\ref{proposition:perfect.matchings.in.compact.groups}, there exists a finite non-empty subset $F \subseteq H$ such that $\mu (F,hF,\mathcal{V}) = \vert F \vert$ for all $h \in H$. We argue that $\mu (F,gF,\mathcal{U}) = \vert F \vert$ for every $g \in E$. To this end, let $g \in E$. By hypothesis on $H$, there exists $h \in H$ with $gh^{-1} \in V$. Suppose $\phi$ to be a perfect matching in $\mathcal{B}(F,gF,\mathcal{V})$. Define $\psi \colon F \to gF, \, x \mapsto gh^{-1}\phi (x)$. Evidently, $\psi$ is bijective. We show that $\psi$ constitutes a matching in $\mathcal{B}(F,gF,\mathcal{U})$. For this, let $x \in F$. Then there exists $y \in H$ with $\{ x, \phi (x) \} \subseteq (V \cap H)y \subseteq Vy$. Hence, $\phi (x) \in VV^{-1}x$ and thus $\psi (x) \in V\phi(x) \subseteq VVV^{-1}x \subseteq \St (x,\mathcal{U})$. Therefore, $\mu (F,gF,\mathcal{U}) = |F|$. This readily completes the proof. \end{proof}

Finally, we are going to investigate topological matching properties of extremely amenable groups, i.e., we draw a connection between Theorem~\ref{theorem:amenable.groups} and a characterization of extremely amenable topological groups due to Pestov \cite{pestov1} (see also \cite{pestov2,pestovbook}). Recall that topological group $G$ is said to be \emph{extremely amenable} if every continuous action of $G$ on a non-empty compact Hausdorff space admits a fixed point. In order to state and discuss Pestov's result, let us recall some additional terminology from \cite{pestovbook}.

\begin{definition}\label{definition:ramsey} We say that a topological group $G$ has the \emph{Ramsey-Dvoretzky-Milman property} if, for all $\epsilon > 0$ and $f \in {\rm UC}_{b}(G_{r})$ and every finite subset $E \subseteq G$, there exists some $g \in G$ such that $\diam f(Eg) \leq \epsilon$. \end{definition}

For the sake of convenience, let us furthermore mention the subsequent slight, but useful reformulation of the Ramsey-Dvoretzky-Milman property.

\begin{prop}[\cite{pestovbook}]\label{proposition:ramsey} A topological group $G$ has the Ramsey-Dvoretzky-Milman property if and only if, for every $\epsilon > 0$, every finite subset $H \subseteq {\rm UC}_{b}(G_{r})$ and every finite subset $E \subseteq G$, there exists $g \in G$ such that $\diam f(Eg) \leq \epsilon$ for each $f \in H$. \end{prop}

The following result reveals the link between the Ramsey-Dvoretzky-Milman property and extreme amenability.

\begin{thm}[\cite{pestovbook}]\label{theorem:ramsey} A topological group is extremely amenable if and only if it has the Ramsey-Dvoretzky-Milman property. \end{thm}

Now let us restate Pestov's result in terms of finite uniform coverings.

\begin{cor}\label{corollary:extremely.amenable.groups} A topological group $G$ is extremely amenable if and only if, for every finite subset $E \subseteq G$ and each $\mathcal{U} \in \mathcal{N}(G_{r})$, there exist $g \in G$ and $U \in \mathcal{U}$ such that $Eg \subseteq U$. \end{cor}

\begin{proof} ($\Longrightarrow $) Let $E \subseteq G$ be finite and let $\mathcal{U} \in \mathcal{N}(G_{r})$. Without loss of generality, assume $E$ to be non-empty. By Lemma~\ref{lemma:uniform.covers}, there exists a family of uniformly continuous functions $f_{U} \colon G_{r} \to [0,1]$ ($U \in \mathcal{U}$) such that \begin{enumerate}
	\item[(1)]	$\spt (f_{U}) \subseteq U$ for every $U \in \mathcal{U}$,
	\item[(2)]	$\sum_{U \in \mathcal{U}} f_{U} = 1$.
\end{enumerate} Due to Theorem~\ref{theorem:ramsey}, $G$ has the Ramsey-Dvoretzky-Milman property. Hence, by Proposition~\ref{proposition:ramsey}, there exists $g \in G$ such that $\diam f(Eg) \leq \frac{1}{|\mathcal{U}| + 1}$ for each $f \in H$. Let $h_{0} \in E$. By (2), there exists $U \in \mathcal{U}$ such that $f_{U}(h_{0}g) > \frac{1}{|U| + 1}$. We conclude that $f_{U}(hg) > 0$ for each $h \in E$. Consequently, $Eg \subseteq U$ due to (1). This proves the claim.

($\Longleftarrow $) Let $\epsilon > 0$, $f \in {\rm UC}_{b}(G_{r})$ and $E \in \mathcal{F}(G)$. Due to Lemma~\ref{lemma:uniform.tiles}, there exists $\mathcal{U} \in \mathcal{N}(G_{r})$ such that $\diam f(U) \leq \epsilon$ for all $U \in \mathcal{U}$. By assumption, there exist $g \in G$ and $U \in \mathcal{U}$ such that $Eg \subseteq U$. Hence, $\diam f(Eg) \leq \epsilon$. this completes the proof. \end{proof}

Let us briefly discuss the connection between Theorem~\ref{theorem:amenable.groups} and Corollary~\ref{corollary:extremely.amenable.groups}. To this end, suppose $G$ to be an extremely amenable topological group. Let $\mathcal{U} \in \mathcal{N}(G_{r})$ and let $E$ be a finite subset of $G$. Due to Corollary~\ref{corollary:extremely.amenable.groups}, there exist $g \in G$ and $U \in \mathcal{U}$ such that $(E \cup \{ e \})g \subseteq U$. Let $F \defeq \{ g \}$. Then $F \cup hF \subseteq \{ g, hg \} \subseteq U$ and thus $\mu (F,hF,\mathcal{U}) = 1 = |F|$ for each $h \in E$. This shows that $\mu (G) = 1$. In particular, $G$ is amenable by Theorem~\ref{theorem:amenable.groups}.

\vspace{0.2cm}

Let us finish this section with some application to the theory of von Neumann algebras, i.e., unital, weakly closed, self-adjoint subalgebras of the algebra of bounded operators on a Hilbert space $\mathcal H$. By definition $N \subset B(\mathcal H)$ is called injective if any completely positive linear map from any self adjoint closed subspace containing $1$ of any unital $C^\ast$-algebra $A$ to $N$ can be extended to a completely positive map from $A$ to $M$. Through the seminal work of Connes \cite{connes}, it is known that $N$ is injective if and only if it is approximately finite dimensional. Our characterization of injectivity of the algebra $N$ is in terms of a perfect matching condition for its unitary group.

\begin{cor}
A von Neumann algebra $N$ is injective if and only if its unitary group $G:=U(N)$ satisfies the following condition. If $E \subseteq G$ is finite and $\mathcal{U}$ is a finite uniform covering of $G_{r}$, then there exists $F \in \mathcal{F}_{+}(G)$ such that $\mu (F,gF,\mathcal{U}) = |F|$ for all $g \in E$.
\end{cor}
\begin{proof}
By \cite[Theorem 3.3]{giordanopestov}, the unitary of any approximately finite dimensional von Neumann algebra is a product of a compact group and an extremely amenable group. One direction then follows from Corollaries \ref{proposition:perfect.matchings.in.compact.groups} and \ref{corollary:extremely.amenable.groups}. On the other side, the perfect matching condition clearly implies amenability of $G$, and hence that $N$ is injective by \cite{delaharpe}.
\end{proof}

\section{A Ramsey condition for metric Fra\"iss\'e classes}\label{section:ramsey.theory}

This section shall be devoted to reformulating Theorem~\ref{theorem:two.element.coverings} for metric Fra\"iss\'e structures in the context of continuous logic (cf.~\cite{YaacovUsvyatsov,yaacov}). In recent years, the connection between the combinatorics of Fra\"iss\'e classes and the topological dynamics of the automorphism groups of their Fra\"iss\'e limits has attracted a lot of interest. In \cite{KechrisPestovTodorcevic} Kechris, Pestov, and Todor\v{c}evi\'c showed that the Ramsey property for a Fra\"iss\'e order class is equivalent to the automorphism group of its Fra\"iss\'e limit being extremely amenable. This result was extended to continuous logic by Melleray and Tsankov \cite{melleray}. In \cite{moore} Moore established a corresponding equivalence between the convex Ramsey property and amenability, which was generalized to the setting of continuous logic by Ka{\"i}chouh \cite{kaichouh}. Note that -- apart from the results mentioned so far -- a similar correspondence between the \emph{Hrushovski property} and compact approximability was proven by Kechris and Rosendal \cite{KechrisRosendal}.

We need to recall some notation and terminology from \cite{melleray}. So, let $\mathcal{L}$ be a \emph{language}, i.e., a set of \emph{relational symbols}, to each of which there is associated an \emph{arity} (a positive natural number) and a \emph{Lipschitz constant} (a non-negative real number). An \emph{$\mathcal{L}$-structure} $\mathbf{A}$ consists of a complete metric space $(A,d)$ along with an $l$-Lipschitz continuous function $R^{\mathbf{A}} \colon A^{n} \to \mathbb{R}$ for each $n$-ary relational symbol $R \in \mathcal{L}$ with Lipschitz constant $l$. For this to make sense, we need to say that we always endow finite products of metric spaces the supremum metric. An $\mathcal{L}$-structure is called \emph{Polish} if the underlying metric space is separable. Let $\mathbf{A}$ and $\mathbf{B}$ be $\mathcal{L}$-structures. A \emph{morphism} (or \emph{embedding}) from $\mathbf{A}$ to $\mathbf{B}$ is an isometric map $\alpha \colon A \to B$ such that $R^{\mathbf{A}}(a_1,\ldots,a_n) = R^{\mathbf{B}}(\alpha(a_{1}),\ldots,\alpha (a_{n}))$ for every $n$-ary symbol $R \in \mathcal{L}$ and every $a=(a_1,\ldots,a_n) \in A^{n}$. We call $\mathbf{B}$ a \emph{substructure} of $\mathbf{A}$ and write $\mathbf{B} \leq \mathbf{A}$ if $B \subseteq A$ and the natural injection $B \to A, \, b \mapsto b$ is a morphism from $\mathbf{B}$ to $\mathbf{A}$. The set ${}^{\mathbf{A}}\mathbf{B}$ of all morphisms from $\mathbf{A}$ to $\mathbf{B}$ comes along with a metric given by \begin{displaymath}
	\rho_{\mathbf{A}}(\alpha,\beta) \defeq \sup_{b \in B} d(\alpha (b),\beta (b)) \qquad (\alpha, \beta \in {}^{\mathbf{A}}\mathbf{B}) .
\end{displaymath} Moreover, we shall be concerned with the topological group $\Aut (\mathbf{A})$ of all automorphisms of $\mathbf{A}$ endowed with the topology of pointwise convergence. It is easy to see that $\Aut (\mathbf{A})$ is a closed subgroup of the isometry group of the underlying metric space of $\mathbf{A}$. Hence, if $\mathbf{A}$ is Polish, then $\Aut (\mathbf{A})$ is a Polish group.

Now we come to \emph{Fra\"iss\'e classes}. For a precise definition of this term, we refer to \cite{melleray}. For our purposes, the description of Fra\"iss\'e classes as \emph{ages} of \emph{homogeneous} Polish structures is sufficient. To give a bit more detail, let again $\mathcal{L}$ be a language and $\mathbf{A}$ be an $\mathcal{L}$-structure. We say that $\mathbf{A}$ is \emph{homogeneous} if, for every $\epsilon > 0$, every finite substructure $\mathbf{B} \leq \mathbf{A}$ and every morphism $\beta \colon \mathbf{B} \to \mathbf{A}$, there exists an automorphism $\alpha$ of $\mathbf{A}$ such that $\rho_{\mathbf{B}}(\alpha|_{\mathbf{B}},\beta) < \epsilon$, i.e., $d (\alpha(b),\beta(b)) < \epsilon$ for all $b \in B$. The \emph{age} of $\mathbf{A}$ is the class of all finite $\mathcal{L}$-structures which embed into $\mathbf{A}$, i.e., which admit a morphism to $\mathbf{A}$. Now the following correspondence between homogeneous Polish structures and Fra\"iss\'e classes of finite structures holds.

\begin{thm}[\cite{yaacov}] Let $\mathcal{L}$ be a language. A Polish $\mathcal{L}$-structure $\mathbf{A}$ is homogeneous if and only if the age of $\mathbf{A}$ is a Fra\"iss\'e class. \end{thm}

\begin{thm}[\cite{yaacov}] Let $\mathcal{L}$ be a language and $\mathcal{K}$ a Fra\"iss\'e class of finite $\mathcal{L}$-structures. There is (up to isomorphism) a unique homogeneous $\mathcal{L}$-structure whose age is equal to $\mathcal{K}$. This $\mathcal{L}$-structure, called the \emph{Fra\"iss\'e limit} of $\mathcal{K}$, is Polish. \end{thm}

In the light of the above correspondence, it seems natural to ask for a characterization of amenability for automorphism groups of Fra\"iss\'e limits in terms of combinatorial properties of the respective Fra\"iss\'e classes. We will provide such a characterization. For this purpose, we need to investigate a certain class of bipartite graphs in the context of $\mathcal{L}$-structures for a fixed language $\mathcal{L}$. To explain this, let $\epsilon > 0$ and consider finite $\mathcal{L}$-structures $\mathbf{A}$, $\mathbf{B}$, $\mathbf{C}$, embeddings $\alpha ,\beta \in {}^{\mathbf{A}}\mathbf{B}$, a map $\psi \colon F \to {}^{\mathbf{B}}\mathbf{C}$ with a finite domain $F$, and a map $\phi \colon {}^{\mathbf{A}}\mathbf{C} \to \{ 0,\ldots,k \}$ with $k \geq 1$. Let us consider the bipartite graph \begin{displaymath}
	\mathcal{B}(\psi,\alpha,\beta,\phi,\epsilon) \defeq (F,F,\{ (\gamma,\gamma') \in F^{2} \mid \exists i \in \{ 0,\ldots,k \} \colon \, \{ \psi (\gamma ) \alpha, \psi (\gamma') \beta \} \subseteq B_{\rho_{\mathbf{A}}}(\phi^{-1}(i),\epsilon) \}) 
\end{displaymath} and its matching number $\mu (\psi,\alpha,\beta,\phi,\epsilon) \defeq \mu (\mathcal{B}(\psi,\alpha,\beta,\phi,\epsilon))$ in particular. Utilizing this notation, we can reformulate Theorem~\ref{theorem:two.element.coverings} as follows.

\begin{thm}\label{theorem:metric.ramsey} Let $\mathcal{L}$ be a language, let $\mathcal{K}$ be a Fra\"iss\'e class of finite $\mathcal{L}$-structures and $\mathbf{K}$ its Fra\"iss\'e limit. For every integer $k\geq 1$, the following are equivalent: \begin{enumerate}
	\item[$(1)$] $\Aut (\mathbf{K})$ is amenable.
	\item[$(2)$] For every $\epsilon > 0$ and any two structures $\mathbf{A},\mathbf{B} \in \mathcal{K}$, there is a structure $\mathbf{C} \in \mathcal{K}$ such that, for every coloring $\phi \colon {}^{\mathbf{A}}\mathbf{C} \to \{ 0,\ldots,k \}$, there exists a map $\psi \colon F \to {}^{\mathbf{B}}\mathbf{C}$ with a finite non-empty domain $F$ such that \begin{displaymath}
			\forall \alpha, \beta \in {}^{\mathbf{A}}\mathbf{B} \colon \, \mu(\psi,\alpha,\beta,\phi,\epsilon) \geq (1-\epsilon)\vert F \vert .
		\end{displaymath}
\end{enumerate} \end{thm}

\begin{proof} (1)$\Longrightarrow$(2). Let $\epsilon > 0$ and $\mathbf{A},\mathbf{B} \in \mathcal{K}$. Without loss of generality, we assume that $\mathbf{A}$ and $\mathbf{B}$ are finite substructures of $\mathbf{K}$. Since $\mathbf{K}$ is homogeneous, we may choose an injective map $\gamma \colon {}^{\mathbf{A}}\mathbf{B} \to \Aut (\mathbf{K})$ such that $\rho_{\mathbf{A}}(\alpha,\gamma(\alpha)|_{\mathbf{A}}) \leq \frac{\epsilon}{2}$ for all $\alpha \in {}^{\mathbf{A}}\mathbf{B}$. Let $E \defeq \gamma ({}^{\mathbf{A}}\mathbf{B})$ and consider the identity neighborhood $U \defeq \{ g \in G \mid \forall a \in A \colon \, d(a,g(a)) < \frac{\epsilon}{4} \}$ in $G$. As $G \defeq \Aut (\mathbf{K})$ is amenable, Theorem~\ref{theorem:two.element.coverings}~(4) states the following: there exists a finite subset $S \subseteq G$ such that, for every two-element $U$-uniform covering $U$ of $G_{l}$, there exists some non-empty subset $F \subseteq S$ such that $FE \subseteq S$ and \begin{displaymath}
	\forall g,h \in E \colon \, \mu (Fg,Fh,\mathcal{U}) \geq (1-\epsilon)\vert F \vert .
\end{displaymath} Denote by $\mathbf{C}$ the substructure of $\mathbf{K}$ with domain $C \defeq \{ g(b) \mid b \in B, \, g \in S \}$. We claim that $\mathbf{C}$ has the desired property. So, let $\phi \colon {}^{\mathbf{A}}\mathbf{C} \to \{ 0,\ldots,k \}$. Let $T \defeq \{ g \in G \mid g(A) \nsubseteq C \}$ and \begin{displaymath}
	V_{i} \defeq \{ g \in G \mid \exists \alpha \in \phi^{-1}(i) \colon \, \rho_{\mathbf{A}}(g|_{\mathbf{A}},\alpha) < \tfrac{\epsilon}{2} \} \cup T
\end{displaymath} for $i \in \{ 0,\ldots,k \}$. Note that $\mathcal{V} \defeq \{ V_{0},\ldots,V_{k} \}$ is a $U$-uniform covering of $G_{l}$: for each $g \in G$, either $\rho_{\mathbf{A}}(g|_{\mathbf{A}},\alpha) < \tfrac{\epsilon}{4}$ for some $\alpha \in {}^{\mathbf{A}}\mathbf{C}$ and hence $gU \subseteq V_{i}$ for some $i \in \{ 0,\ldots,k \}$, or $\rho_{\mathbf{A}}(g|_{\mathbf{A}},\alpha) \geq \tfrac{\epsilon}{4}$ for all $\alpha \in {}^{\mathbf{A}}\mathbf{C}$ and thus $gU \subseteq T \subseteq V_{i}$ for any $i \in \{ 0,\ldots,k \}$. By hypothesis on $S$, there exists a non-empty subset $F \subseteq S$ such that $FE \subseteq S$ and \begin{displaymath}
	\forall g,h \in E \colon \, \mu (Fg,Fh,\mathcal{V}) \geq (1-\epsilon)\vert F \vert .
\end{displaymath} Consider the map $\psi \colon F \to {}^{\mathbf{B}}\mathbf{C}, \, g \mapsto g|_{\mathbf{B}}$. For all $\alpha,\alpha' \in {}^{\mathbf{A}}\mathbf{C}$ and $x \in F\gamma (\alpha)$, we observe that \begin{align*}
\rho_{\mathbf{A}}(x|_{\mathbf{A}},\alpha') < \tfrac{\epsilon}{2} \ \ &\Longleftrightarrow \ \ \rho_{\mathbf{A}}(\gamma (\alpha)|_{\mathbf{A}},\gamma (\alpha)x^{-1}\alpha') < \tfrac{\epsilon}{2} \\
&\Longrightarrow \ \ \rho_{\mathbf{A}}(\alpha,\gamma (\alpha)x^{-1}\alpha') < \epsilon \ \ \Longleftrightarrow \ \ \rho_{\mathbf{A}}(x\gamma (\alpha)^{-1}\alpha,\alpha') < \epsilon .
\end{align*} Now, let $\alpha,\beta \in {}^{\mathbf{A}}\mathbf{B}$. For all $x \in F\gamma (\alpha)$, $y \in F\gamma (\beta)$, and $i \in \{ 0,\ldots,k\}$, we have \begin{align*}
	\{ x,y \} \subseteq V_{i} \ &\stackrel{x,y \in S \subseteq G\setminus T}{\Longleftrightarrow} \ \exists \alpha',\beta' \in \phi^{-1}(i) \colon \, \rho_{\mathbf{A}}(x|_{\mathbf{A}},\alpha') < \tfrac{\epsilon}{2}, \, \rho_{\mathbf{A}}(y|_{\mathbf{A}},\beta') < \tfrac{\epsilon}{2} \\
	&\Longrightarrow \ \exists \alpha',\beta' \in \phi^{-1}(i) \colon \, \rho_{\mathbf{A}}(x\gamma (\alpha)^{-1}\alpha,\alpha') < \epsilon, \, \rho_{\mathbf{A}}(y\gamma (\beta)^{-1}\beta,\beta') < \epsilon \\
	&\Longleftrightarrow \ \{ x\gamma (\alpha)^{-1}\alpha , \, x\gamma (\beta)^{-1}\beta \} \subseteq B_{\rho_{\mathbf{A}}}(\phi^{-1}(i),\epsilon) \\
	&\Longleftrightarrow \ \{ \psi (x\gamma (\alpha)^{-1}) \alpha , \, \psi (x\gamma (\beta)^{-1}) \beta \} \subseteq B_{\rho_{\mathbf{A}}}(\phi^{-1}(i),\epsilon).
\end{align*} Applying Remark~\ref{remark:bijective.graph.homomorphisms} with respect to the bijective mappings $\lambda_{\alpha} \colon F\gamma (\alpha) \to F, \, x \mapsto x\gamma (\alpha)^{-1}$ and $\lambda_{\beta} \colon F\gamma (\beta) \to F, \, x \mapsto x\gamma (\beta)^{-1}$, we conclude that \begin{displaymath}
	\mu(\psi,\alpha ,\beta ,\phi ,\epsilon) \geq \mu (F\gamma(\alpha),F\gamma (\beta),\mathcal{V}) \geq (1-\epsilon)\vert F \vert = (1-\epsilon)\vert F' \vert .
\end{displaymath} 

(2)$\Longrightarrow$(1). Suppose that (2) holds for some $k \geq 1$.  It follows that (2) is valid for $k=1$. We utilize Theorem~\ref{theorem:two.element.coverings} and show that $G$ satisfies (3) in Theorem~\ref{theorem:two.element.coverings} for $k=1$ and $\theta = \frac{1}{2}$. Consider a finite subset $E \subseteq G$ and a two-element uniform covering $\mathcal{U}$ of $G_{l}$. Choose some $\epsilon > 0$ and a finite subset $A \subseteq K$ such that $\mathcal{U} \preceq \{ gU \mid x \in G \}$ with regard to the identity neighborhood $U \defeq \{ g \in G \mid \forall a \in A \colon \, d(a,g(a)) < \epsilon \}$ in $G$. As $\vert \mathcal{U} \vert \leq 2$, there is a subset $H \subseteq G$ so that $\mathcal{U}$ is refined by $\mathcal{V} \defeq \{ HU, \, (G\setminus H)U \}$. We denote by $\mathbf{A}$ the substructure of $\mathbf{K}$ with domain $A$ and by $\mathbf{B}$ the substructure of $\mathbf{K}$ with domain $B \defeq \{ g(a) \mid a \in A, \, g \in E \}$. By assumption, there exists a structure $\mathbf{C} \in \mathcal{K}$ such that, for every coloring $\phi \colon {}^{\mathbf{A}}\mathbf{C} \to \{ 0,1 \}$, there is a map $\psi \colon F \to {}^{\mathbf{B}}\mathbf{C}$ with a finite non-empty domain $F$ such that \begin{displaymath}
	\forall \alpha, \beta \in {}^{\mathbf{A}}\mathbf{B} \colon \, \mu(\psi, \alpha ,\beta ,\phi,\tfrac{\epsilon}{3}) \geq \tfrac{1}{2}\vert F \vert .
\end{displaymath} Without loss of generality, we may assume that $\mathbf{C}$ is a finite substructure of $\mathbf{K}$. Since $\mathbf{K}$ is homogeneous, we find an injective map $\gamma_{0} \colon {}^{\mathbf{A}}\mathbf{C} \to \Aut (\mathbf{K})$ such that $\rho_{\mathbf{A}}(\alpha,\gamma_{0}(\alpha)) \leq \frac{\epsilon}{3}$ for all $\alpha \in {}^{\mathbf{A}}\mathbf{C}$. Consider the coloring $\phi \defeq 1_{H} \circ \gamma_{0} \colon {}^{\mathbf{A}}\mathbf{C} \to \{ 0,1 \}$. By hypothesis, there is a map $\psi \colon F \to {}^{\mathbf{B}}\mathbf{C}$ with a finite non-empty domain $F$ such that \begin{displaymath}
	\forall \alpha, \beta \in {}^{\mathbf{A}}\mathbf{B} \colon \, \mu(\psi, \alpha ,\beta ,\phi,\tfrac{\epsilon}{3}) \geq \tfrac{1}{2}\vert F \vert .
\end{displaymath} Again due to homogeneity of $\mathbf{K}$, there exists an injective map $\gamma_{1} \colon F \to \Aut (\mathbf{K})$ such that $\rho_{\mathbf{B}}(\psi(\alpha),\gamma_{1}(\alpha)) \leq \frac{\epsilon}{3}$ for all $\alpha \in F$. Let $F' \defeq \gamma_{1}(F)$. We show that $\mu (F'g,F'h,\mathcal{U}) \geq \tfrac{1}{2}\vert F' \vert$ for any two $g,h \in E$. To this end, let us first observe that \begin{displaymath}
	\rho_{\mathbf{A}}(\psi(\alpha) g|_{\mathbf{A}},\alpha') < \tfrac{\epsilon}{3} \ \Longrightarrow \ \rho_{\mathbf{A}}(\gamma_{1}(\alpha) g|_{\mathbf{A}},\alpha') < \tfrac{2\epsilon}{3} \ \Longrightarrow \ \rho_{\mathbf{A}}(\gamma_{1}(\alpha) g|_{\mathbf{A}},\gamma_{0}(\alpha')|_{\mathbf{A}}) < \epsilon
\end{displaymath} for all $g \in E$, $\alpha \in F$, and $\alpha' \in {}^{\mathbf{A}}\mathbf{C}$. Now, let $g,h \in E$. For all $\alpha, \beta \in F$ and $L \in \{ H, \, G\setminus H \}$, \begin{align*}
	\{ \psi(\alpha) g|_{\mathbf{A}},\, &\psi(\beta) h|_{\mathbf{A}} \} \subseteq B(\gamma_{0}^{-1}(L),\epsilon/3) \\
	&\Longleftrightarrow \ \exists \alpha',\beta' \in \gamma_{0}^{-1}(L) \colon \, \rho_{\mathbf{A}}(\psi(\alpha) g|_{\mathbf{A}},\alpha') < \tfrac{\epsilon}{3}, \, \rho_{\mathbf{A}}(\psi(\beta) h|_{\mathbf{A}},\beta') < \tfrac{\epsilon}{3} \\
	&\Longrightarrow \ \exists \alpha',\beta' \in \gamma_{0}^{-1}(L) \colon \, \rho_{\mathbf{A}}(\gamma_{1}(\alpha) g|_{\mathbf{A}},\gamma_{0}(\alpha')|_{\mathbf{A}}) < \epsilon, \, \rho_{\mathbf{A}}(\gamma_{1}(\beta) h|_{\mathbf{A}},\gamma_{0}(\beta')|_{\mathbf{A}}) < \epsilon \\
	&\Longrightarrow \ \exists k,l \in L \colon \, \rho_{\mathbf{A}}(\gamma_{1}(\alpha) g|_{\mathbf{A}},k|_{\mathbf{A}}) < \epsilon, \, \rho_{\mathbf{A}}(\gamma_{1}(\beta) h|_{\mathbf{A}},l|_{\mathbf{A}}) < \epsilon \\
	&\Longleftrightarrow \ \{ \gamma_{1}(\alpha) g , \, \gamma_{1}(\beta) h \} \subseteq LU.
\end{align*} Applying Remark~\ref{remark:bijective.graph.homomorphisms} with regard to the bijective mappings $\psi_{g} \colon F \to F'g, \, \alpha \mapsto \gamma_{1}(\alpha) g$ and $\psi_{h} \colon F \to F'h, \, x \mapsto \gamma_{1}(\alpha) h$, we conclude that \begin{displaymath}
	\mu (F'g,F'h,\mathcal{U}) \geq \mu (F'g,F'h,\mathcal{V}) \geq \mu (F,g|_{\mathbf{A}},h|_{\mathbf{A}},\phi,\epsilon) \geq \tfrac{1}{2}\vert F \vert = \tfrac{1}{2}\vert F' \vert . \qedhere
\end{displaymath} \end{proof}

We hope that this characterization will be useful both in the study of Fra\"iss\'e classes whose associated automorphism group is amenable, and the study of amenability of particular polish groups that appear as the automorphism groups of  some tractable Fra\"iss\'e classes.

\begin{exmpl} Let us finish with the discussion of a particular example. Let ${\rm Aff}(\mathcal H) = \mathcal H \rtimes O(\mathcal H)$ be the group of affine isometries of a (real) infinite-dimensional Hilbert space. The associated 
Fra\"iss\'e classes is the class of finite metric spaces that are isometrically embeddable into a Hilbert space.
It is easy to see that ${\rm Aff}(\mathcal H)$ is amenable but not extremely amenable, so that Theorem \ref{theorem:metric.ramsey} can give some understanding of the metric combinatorics of the class of these metric spaces. Indeed, in the simplest case $\mathbf{A} = \ast$ and $\mathbf{B}$ is some finite subset of $\mathcal H$. Let $k$ and $\varepsilon>0$ be fixed. Now, according to Theorem \ref{theorem:metric.ramsey}, there exists some larger finite subset $\mathbf{C}$ (depending on $\mathbf B$, $\varepsilon$ and $k$), such that for any coloring $\varphi \colon \mathbf{C} \to \{0,\ldots,k\}$, there exists a finite multi-set $F$ of isometric embeddings of $\mathbf{B}$ into $\mathbf{C}$, such that the following holds: For any pair of points of $\alpha,\beta \in \mathbf{B}$,
at least $(1-\varepsilon)|F|$ of the images of $\alpha$ in $\mathbf{C}$ can be matched with images of $\beta$ in $\mathbf{C}$, so that matched pairs lie $\varepsilon$-close to a color-class.

The notion of extreme amenability (which is more classical in this context) would assert all of $\mathbf B$ could be mapped $\varepsilon$-close to a single color class. This applies for example in the case of the Fra\"iss\'e class of finite metric spaces that embed isometrically into the unit sphere of $\mathcal H$.
\end{exmpl}

\section*{Acknowledgments}

This research was supported by ERC-Starting Grant 277728. Furthermore, the first author is supported by funding of the Excellence Initiative by the German Federal and State Governments.

%%%%%%%%%%%%%%%%%%%%%%%%%%%%%%%%%%

%\small

\bibliographystyle{amsalpha}

\def\cprime{$'$}

\end{document}